\documentclass[english]{amsart}

\usepackage{enumerate}
\usepackage{amssymb}
\usepackage{array}
\usepackage{graphicx}
\usepackage{mathrsfs}
\usepackage{multicol}
\usepackage[a4paper,margin=2.7cm]{geometry}
\usepackage{mathtools}
\usepackage{ mathdots }
\usepackage{color}
\usepackage[dvipsnames,svgnames,table]{xcolor}
\usepackage{comment}
\usepackage[normalem]{ulem}
\usepackage{pgf,tikz,pgfplots}
\usepackage{tikz-cd}
\usetikzlibrary{arrows}
\pgfplotsset{compat=1.15}

\usepackage[curve,all]{xy}
\xyoption{matrix}
 \xyoption{curve}
 \xyoption{color}
 \xyoption{line}
\xyoption{arc}

\usepackage[shortlabels]{enumitem}
\setlist[enumerate]{leftmargin=7mm,topsep=0pt,itemsep=-1ex,partopsep=1ex,parsep=1ex,label=\rm{(\roman*)}}
\setlist[itemize]{leftmargin=5mm,topsep=0pt,itemsep=-1ex,partopsep=1ex,parsep=1ex,label=\raisebox{0.25ex}{\tiny$\bullet$}}

\usepackage[backref, colorlinks, linktocpage, citecolor = blue, linkcolor = blue]{hyperref}
\usepackage{cleveref}

\theoremstyle{plain}
\newtheorem{theorem}{Theorem}[section]

\newtheorem*{theoremaux}{Theorem \theoremauxnum}
\gdef\theoremauxnum{1}

\newtheorem*{main-theorem}{Main Theorem}
\newtheorem{proposition}[theorem]{Proposition}
\newtheorem*{propositionaux}{Proposition \propositionauxnum}
\gdef\propositionauxnum{1}

\newtheorem{lemma}[theorem]{Lemma}
\newtheorem*{lemmaaux}{Lemma \lemmaauxnum}
\gdef\lemmaauxnum{1}

\newtheorem{corollary}[theorem]{Corollary}

\newtheorem*{key-problem}{Key Problem}
\newtheorem*{theorem-intro}{Theorem}
\newtheorem*{proposition-intro}{Proposition}
\newtheorem*{corollary-intro}{Corollary}

\theoremstyle{definition}
\newtheorem{definition}[theorem]{Definition}

\newtheorem{example}[theorem]{Example}

\theoremstyle{remark}
\newtheorem{remark}[theorem]{Remark}

\makeatletter
\@namedef{subjclassname@2020}{%
  \textup{2020} Mathematics Subject Classification}
\makeatother
\newcommand{\incl}[1][r]{\ar@<-0.2pc>@{^(-}[#1] \ar@<+0.2pc>@{-}[#1]}
\newcommand{\hs}{\kern 0.8pt}

\renewcommand{\H}{{\mathrm{H}}}

\renewcommand{\P}{\mathbb{P}}

\newcommand{\Loc}{\mathrm{Loc}}

\newcommand{\Hom}{\mathrm{Hom}}

\renewcommand{\div}{\mathrm{div}}
\newcommand{\E}{\mathcal{E}}

\renewcommand{\L}{\mathcal{L}}

\renewcommand{\H}{\mathrm{H}}
\newcommand{\Supp}{\mathrm{Supp}}
\newcommand{\Spec}{\mathrm{Spec}}
\newcommand{\Stab}{\mathrm{Stab}}

\newcommand{\Q}{\mathbb{Q}}

\renewcommand{\H}{\mathrm{H}}

\newcommand{\A}{\mathbb{A}}

\newcommand{\D}{\mathfrak{D}}

\newcommand{\N}{\mathbb{N}}

\newcommand{\Z}{\mathbb{Z}}

\newcommand{\Gm}{\mathbb{G}_\mathrm{m}}

\renewcommand{\O}{\mathcal{O}}

\newcommand{\PPDiv}{\mathrm{PPDiv}}
\newcommand{\CaDiv}{\mathrm{CaDiv}}
\newcommand{\Div}{\mathrm{Div}}
\newcommand{\Sd}{\mathfrak{S}}
\newcommand{\Ed}{\mathfrak{E}}

\DeclareMathOperator{\Aut}{Aut}

\DeclareMathOperator{\PGL}{PGL}

\DeclareMathOperator{\Pic}{Pic}

\DeclareMathOperator{\Bir}{Bir}

\title[Classification of equivariantly normal curves via Altmann--Hausen--S\"u\ss{} theory]{Classification of equivariantly normal curves\\ via Altmann--Hausen--S\"u\ss{} theory}
\date{Version of \today}

\author{Gary Martínez-Núñez}
\address{Univ. Lille, CNRS, UMR 8524 - Laboratoire Paul Painlev\'e, F-59000 Lille, France}
\email{gary-antonio.martinez-nunez@univ-lille.fr}

\author{Ronan Terpereau}
\address{Univ. Lille, CNRS, UMR 8524 - Laboratoire Paul Painlev\'e, F-59000 Lille, France}
\email{ronan.terpereau@univ-lille.fr}

\begin{document}

\begin{abstract}
Let the ground field be perfect of positive characteristic.
Using Altmann--Hausen--S\"u\ss{} theory, we obtain a combinatorial classification of equivariantly normal curves with prescribed quotient in both the affine and projective settings.
As a consequence, we derive an explicit upper bound on the number of isomorphism classes of equivariantly normal curves over a fixed base curve with prescribed branch locus.
Furthermore, assuming that the ground field is algebraically closed, we determine, for a fixed cardinality of the branch locus, precisely when the set of isomorphism classes of equivariantly normal projective curves with prescribed quotient is finite and when it is infinite.
\end{abstract}

%14L30   Group actions on varieties or schemes (quotients)
%14L15   Group schemes
%14H37   Automorphisms of curves
%14G17   Positive characteristic ground fields in algebraic geometry

%14E07   Birational automorphisms, Cremona group and generalizations
%14H45   Special algebraic curves and curves of low genus


\maketitle

\tableofcontents
\section{Introduction}

\subsection{Context and motivation} 
The starting point of this work is the following observation. Let $C$ be a smooth projective curve over a perfect field $k$ of characteristic $p>0$. If $G$ is a smooth finite subgroup scheme of $\Bir(C)$, then $G$ is conjugate to a subgroup scheme of $\Aut(C)$. For instance, if $G= \mu_d$ with $\gcd(p,d)=1$ is a subgroup scheme of $\Bir(\P^{1})$, then $G$ is conjugate to a finite subgroup scheme of $\Aut(\P^{1}) \simeq \PGL_{2}$. 
However, if $G$ is non-smooth---for instance, a nontrivial infinitesimal subgroup scheme of $\Bir(C)$---then $G$ need not be conjugate to a subgroup scheme of $\Aut(C)$. This raises the question of how to describe, in practice, the conjugacy classes of non-smooth subgroup schemes of $\Bir(C)$.

A possible answer was given by Brion in \cite{Bri24a} with the notion of an \emph{equivariantly normal} curve.
A $G$-curve $X$ is \emph{$G$-normal} if every finite birational morphism of $G$-curves $f \colon X' \to X$ is an isomorphism.
We refer to \cite[Section~4]{Bri24a} for further background on this notion.
It turns out that the classification of finite subgroup schemes of $\Bir(C)$ reduces to the classification of isomorphism classes of equivariantly normal projective curves birational to $C$; this follows essentially from \cite[Corollary~4.4]{Bri24a}, see also Proposition~\ref{prop: first reduction} for a precise statement.

In this article, following a suggestion of Brion (see \cite[Remark~2.5.3]{Bri26}), we consider the case where $G = \mu_d$, the subgroup of the multiplicative group $\Gm$ defined by the equation $t^d = 1$, and apply Altmann--Hausen--S\"u\ss{} theory to obtain an explicit combinatorial description of the $G$-isomorphism classes of $G$-normal curves with prescribed quotient.
Note that, from the Altmann--Hausen--S\"u\ss{} perspective, it is natural to consider equivariantly normal curves not within a fixed birational class, but rather with a prescribed quotient, which in our situation is a smooth curve.
Let us also mention that this is the perspective adopted in \cite[\S 3.3]{Bri26}, where $G$-normal varieties with a prescribed quotient are studied.

Note that toric varieties were originally developed in the context of the study of the Cremona group (see \cite{Dem70}). This viewpoint establishes a connection between the Cremona group and the broader theory of varieties endowed with an effective torus action, albeit an indirect one. In this regard, the Altmann--Hausen--S\"u\ss{} theory offers a number of tools for investigating $\Bir(X)$ for arbitrary algebraic varieties $X$, and the present work may be regarded as a further step in that direction.

\subsection{Main results and outline of the paper} \label{sec: intro main results}
Let $G \subset T=\Gm$ be a finite subgroup scheme and let $\widetilde{C}$ be a given smooth (projective or affine) algebraic curve.
In Section~\ref{sec: reductions} we start by showing that the classification of $G$-isomorphism classes of $G$-normal curves $C$ with quotient the smooth curve $C/G \simeq \widetilde{C}$ reduces to the classification of certain $T$-surfaces.

\begin{proposition-intro}[{Proposition~\ref{prop: second reduction}, see also \cite[Remark~2.5.3]{Bri26}}] 
Let $T$, $G$, and $\widetilde{C}$ as above. 
Then there is a one-to-one correspondence between
\[  \hspace{-4mm}
\mathcal{E}:=\left\{
\begin{array}{c}
\text{$G$-isomorphism classes of} \\
\text{$G$-normal curves with quotient $\widetilde{C}$}
\end{array}
\right\}
\;\overset{1:1}{\longleftrightarrow}\;
\mathcal{T}:=
\left\{ \begin{array}{c} \text{$T$-isomorphism classes of triples $(S, \Phi, \pi)$}\\ \text{where $S$ is a smooth $T$-surface and} \\ \begin{tikzcd}[row sep=1em, column sep=2em] & S \arrow[dl, two heads, "\text{$T$-eq.}"', "\Phi" near end] \arrow[dr, two heads, "\text{$/T$}"', "\pi" near end, swap] & \\ T/G & & \widetilde{C} \end{tikzcd}\\ \text{with $\Phi^{-1}(\{1\})$ is geometrically integral} \end{array} \right\}.
\]
\end{proposition-intro}

We then use Altmann--Hausen--S\"u\ss{} theory to describe the right-hand side above combinatorially.
More precisely, we replace a $G$-normal curve $C$, with quotient $C/G \simeq  \widetilde{C}$ by the associated $\Gm$-surface $C^{\#}:=\Gm \times^G C$, encode it by a proper polyhedral divisor or a divisorial fan over the curve $\widetilde{C}$, and express the condition that the morphism $C^{\#} \to T/G$ has geometrically integral fiber, namely $C$, in purely combinatorial terms.

Section~\ref{sec: AHS theory reminder} contains a brief overview of the theory of $T$-varieties, where $T$ denotes an algebraic torus.
This theory, originally developed by Altmann, Hausen, and S\"u\ss{} in terms of \emph{proper polyhedral divisors}~\cite{AH06} and \emph{divisorial fans}~\cite{AHS08} over algebraically closed fields of characteristic zero, has recently been extended to arbitrary fields by the first-named author (see~\cite{GaryAH,GaryAHS}).

\smallskip

In Section~\ref{section: affine G-curves}, we use Altmann--Hausen theory to give a combinatorial description of the set $\mathcal{T}$ in the affine setting. Our main result in this direction is the following.

\begin{theorem-intro}[Theorem~\ref{th: affine refinement}]
Let $C_0$ be a smooth affine curve, and let $G=\mu_{{d}}$ be a finite subgroup scheme of $T=\Gm$. 
There is a one-to-one correspondence
\[\E_{\mathrm{aff}}:=
\left\{
\begin{array}{c}
\text{$G$-isomorphism classes of $G$-normal}\\
\text{affine curves with quotient $C_0$}
\end{array}
\right\}
\;\longleftrightarrow\;
\mathcal{S}_{\mathrm{aff}}:=\left\{
\begin{array}{c}
\text{isomorphism classes of geometrically}\\
\text{integral $\Phi$-pairs $(\D,\alpha)$ over $C_0$}
\end{array}
\right\}.
\]
\end{theorem-intro}

We then prove the finiteness of $G$-isomorphism classes of $G$-normal affine curves over a smooth affine curve with prescribed branch locus; this is a special case of \cite[Corollary~3.3.7]{Bri26}, valid in arbitrary dimension. We furthermore provide an explicit upper bound, which is the original part of the statement.

\begin{proposition-intro}[Proposition~\ref{prop: boundedness affine case}]
Let $C_0$ be a smooth affine curve, let $G=\mu_d$ be a finite subgroup scheme of $T=\Gm$, and assume that $k^{*}=(k^{*})^{d}$.
Then there are only finitely many $G$-isomorphism classes of $G$-normal affine curves over $C_0$ with branch locus $Z$. Moreover, their number is bounded by
\[
N:=
\left|(k[C_{0}\setminus Z]^{*}/k[C_{0}]^{*})/
(k[C_{0}\setminus Z]^{*}/k[C_{0}]^{*})^{d}\right|
\cdot
\left|k[C_0]^{*}/(k[C_0]^{*})^{d}\right|.
\]
If, in addition, $k[C_0]$ is a UFD, then one obtains the simpler bound
\[
N':=(d-1)^n\cdot
\left|k[C_0]^{*}/(k[C_0]^{*})^{d}\right|,
\]
where $n$ is the cardinality of $|Z|$.
\end{proposition-intro}

We also explain how to derive equations for the associated normal affine $T$-surface $S=C^{\#}$ (and thus for $C$) from the combinatorial data, and illustrate this procedure with an example (see Section~\ref{sec: Explicit construction of XD from a pp-divisor and examples}).

\smallskip

In Section~\ref{sec: projective section}, we treat the projective case in a similar spirit. Using Altmann--Hausen--S\"u\ss{} theory, we obtain a combinatorial description of the set $\mathcal{T}$ in the projective setting. Our main result in this direction is the following.
 
\begin{theorem-intro}[Theorem~\ref{th: main result AHS descritpion}]\label{theorem main intro divisorial}
Let $\overline{C}$ be a smooth projective curve, and let $G = \mu_{{d}}$ be a finite subgroup scheme of $T=\Gm$. 
There is a one-to-one correspondence
\[ \mathcal{E}:=
\left\{
\begin{array}{c}
\text{$G$-isomorphism classes of $G$-normal}\\
\text{projective curves with quotient $\overline{C}$}
\end{array}
\right\}
\;\longleftrightarrow\;  \mathcal{S}:=
\left\{
\begin{array}{c}
\text{isomorphism classes of geometrically}\\
\text{integral $\Phi$-pairs $(\Sd,\beta)$  over $\overline{C}$}
\end{array}
\right\}.
\] 
\end{theorem-intro}

As in the affine setting, we then establish in the projective setting the finiteness of isomorphism classes of equivariantly normal projective curves over a smooth projective curve with prescribed branch locus. As before, the finiteness statement is a special case of \cite[Corollary~3.3.7]{Bri26}, valid in arbitrary dimension, whereas the explicit upper bound is new.

\begin{proposition-intro}[Proposition~\ref{proposition: finite number fixed branch locus}]
Let $\overline{C}$ be a smooth projective curve, let $G=\mu_d$ be a finite subgroup scheme of $T=\Gm$, and assume that $k^{*}=(k^{*})^{d}$. Then there are only finitely many $G$-isomorphism classes of $G$-normal projective curves over $\overline{C}$ with branch locus $Z$.
Moreover, their number is bounded by
\small
\[
N:=
\min_{c\in\overline{C}}
\left|
(k[\overline{C}\setminus(\{c\}\cup Z)]^{*}/k[\overline{C}\setminus\{c\}]^{*})
\big/
(k[\overline{C}\setminus(\{c\}\cup Z)]^{*}/k[\overline{C}\setminus\{c\}]^{*})^{d}
\right|
\cdot
\left|
k[\overline{C}\setminus\{c\}]^{*}/
(k[\overline{C}\setminus\{c\}]^{*})^{d}
\right|.
\]
\normalsize
If, in addition, there exists a closed point $c\in\overline{C}$ such that
$C_{0}:=\overline{C}\setminus\{c\}$ satisfies that $k[C_{0}]$ is a UFD, then one obtains the simpler bound
\[
N':=(d-1)^{n-1}\cdot
\left|
k[C_{0}]^{*}/(k[C_{0}]^{*})^{d}
\right|,
\]
where $n$ is the cardinality of $|Z \cap C_0|$.
\end{proposition-intro}

Assuming now that the ground field $k$ is algebraically closed, we investigate when the sets 
\begin{equation} \label{eq: def of En}
\E_n:=
\left\{
\begin{array}{c}
\text{$G$-isomorphism classes of $G$-normal projective curves with quotient $\overline{C}$}\\
\text{such that $C\to\overline{C}$ is a $G$-torsor outside a branch locus of cardinal $n$}
\end{array}
\right\}.
\end{equation}
are finite or infinite.

\begin{theorem-intro}[Theorem~\ref{theorem: finiteness of En}]
Assume that the ground field $k$ is algebraically closed.
Let $G=\mu_{{d}}$ be a finite subgroup scheme of $T=\Gm$ with ${{d}} \geq 2$, and denote by $\varphi$ the Euler's totient function.
Let $\overline{C}$ be a smooth projective curve of genus $g$. Then:
\begin{enumerate}
    \item If $g = 0$ and $n \in \{ 0,1\}$, then $\E_n$ is empty.
     \item    If $g=0$ and $n\in\{2,3\}$, then
$|\E_3| < (d-1)^{\,2}$ and 
$|\E_{2}|=\varphi(d)$.
    \item  If $g=0$ and $n\ge4$, then $\E_n$ is infinite.
    \item  If $g\ge1$ and $n=0$, then $  |\E_n|\le {{d}}^{2g}-1$.
    \item  If $g\ge1$ and $n=1$, then $\E_n$ is empty. 
    \item  If $g=1$ and $n\ge 2$, then $\E_n$ is infinite.
    \item  If $g\ge 2$ and $n\ge 2g+2$, then $\E_n$ is infinite.
\end{enumerate}
\end{theorem-intro}

\begin{remark}
The assumption that the ground field is algebraically closed in Theorem~\ref{theorem: finiteness of En} is made primarily for convenience. We expect that the theorem remains valid over arbitrary perfect fields, although we do not address this generality here.
\end{remark}

Finally, we conclude our article with a complete description of the set $\E_2$ over $\P^1$ (see Section~\ref{sec: complete description of E2 and E3}).

\subsection{Two general remarks}

\begin{remark}
The same strategy could, in principle, describe the conjugacy classes of diagonalizable groups in $\Bir(X)$ for an algebraic variety $X$ of dimension $n \ge 1$. The main obstacle is that this reduces to studying $T$-varieties of higher complexity, whose classification is far less tractable and less amenable to explicit computation than the complexity-one varieties encountered for $\Bir(C)$, with $C$ a curve, and $G = \mu_d$.
\end{remark}

\begin{remark}
The Altmann--Hausen--S\"u\ss{} theory extends to arbitrary fields (see
\cite{GaryAH,GaryAHS}), as does the theory of $G$-normal varieties.
However, the proof of Proposition~\ref{proposition: integrality of the fiber}
relies on a result of Mac Lane \cite[Theorem~2]{Ma39}, which is valid only over
perfect fields. Moreover,
Propositions~\ref{prop: boundedness affine case} and~\ref{proposition: finite number fixed branch locus} no longer hold over imperfect fields
(see \cite[\S~3.3]{Bri26}). For these reasons, we restrict our attention to
perfect ground fields.
\end{remark}

\subsection{Notation and setting}\label{sec:notations}
Throughout, let $k$ be a perfect field of characteristic $p>0$.
A \emph{variety} is a separated, geometrically integral scheme of finite type over $k$.
A \emph{curve} (respectively, a \emph{surface}) is a variety of dimension 1 (respectively, 2).
Recall that if $C$ is an algebraic curve, then there exists a unique smooth projective curve $\overline{C}$ in the birational class of $C$, up to isomorphism. 

We say that a group scheme $G$ is \emph{diagonalizable} if it is isomorphic to a closed subgroup scheme of an algebraic torus $\Gm^{l}$. 
If $C$ is a curve and $G$ is an infinitesimal diagonalizable subgroup scheme of $\Bir(C)$, then necessarily
$G = \mu_{q}$ for some $q=p^s$ (see \cite[Lemma~3.7 and Remark~3.8]{Bri24a}). 
In the present article, we will focus on the case where $G = \mu_d$ with $d \in  \mathbb{N}_{\geq 1}$ arbitrary. In other words, in the infinitesimal case, a diagonalizable group scheme acting on a curve is necessarily isomorphic to $\mu_{p^s}$, so we restrict our attention to finite diagonalizable group schemes of the form $\mu_d$. Note however that, for any smooth diagonalizable finite group scheme $D$, there exists a smooth projective curve $C$ such  that $\Aut(C)\simeq D$; see \cite{MV83}.  

Recall that if $G=\mu_d$ acts faithfully on a curve $C$, then the action is generically free.
We denote by $C_{\mathrm{fr}}$ the free locus of the action, namely the set of points with trivial stabilizer.
Equivalently, $C_{\mathrm{fr}}$ is the largest dense open subset of $C$ such that the quotient morphism
$C_{\mathrm{fr}} \longrightarrow C_{\mathrm{fr}}/G$
is a $G$-torsor. The complement of $C_{\mathrm{fr}}/G$ in $C/G$ is the (reduced) branch locus of the quotient morphism $C \to C/G$.

\bigskip

\noindent \textbf{Acknowledgments.}
We are grateful to Michel Brion and Giancarlo Lucchini-Arteche for helpful discussions related to this work and  for their feedback on a preliminary version of the article.
The authors are grateful to R.~van Dobben de Bruyn for pointing out the argument based on Bertini's theorem in response to the author's question on MathOverflow, which led to the proof of Theorem~\ref{theorem: finiteness of En}~\ref{theorem: finiteness of En part g>1 n>2}.

The authors acknowledge the support of the CDP C2EMPI, as well as the French State under the France-2030 programme, the University of Lille, the Initiative of Excellence of the University of Lille, the European Metropolis of Lille for their funding and support of the R-CDP-24-004-C2EMPI project.

\section{Reduction to the classification of certain \texorpdfstring{$\Gm$-surfaces}{Gm-surfaces}}\label{sec: reductions}
Let $X$ be a variety over the perfect field $k$. The contravariant functor
\[
\mathfrak{Bir}_X \colon (\text{$k$-schemes}) \to (\text{groups})
\]
associates to each $k$-scheme $S$ the group of $S$-families of birational transformations of $X$ (see \cite[\S~2]{Dem70} or \cite[\S~2.1]{BF13} for details). 
This group functor is not representable in general; however, it contains representable subgroup functors. For instance, if $X$ is complete, the automorphism functor $\mathfrak{Aut}_X$ is representable by a $k$-group scheme locally of finite type (see \cite{MO67}).  
A group scheme $G$ representing a subgroup functor $\underline{G} \hookrightarrow \mathfrak{Bir}_X$ is called a \emph{subgroup scheme} of $\Bir(X)$.
Note that, although the abstract groups $\mathfrak{Bir}_{C}(k)$ and $\mathfrak{Aut}_C(k)$ are isomorphic, for a given projective curve $C$, the group functors $\mathfrak{Bir}_{C}$ and $\mathfrak{Aut}_C$ are not.

\smallskip

We can recall the following observation, which will not be used later in this article, but which motivated our study of equivariantly normal curves.

\begin{proposition}[{consequence of \cite[Corollary~4.4]{Bri24a}}] \label{prop: first reduction}
Let $\overline{C}$ be a smooth projective curve over $k$.
Let $G$ be a finite group scheme. Then there is a one-to-one correspondence between 
\[
\left\{
\begin{array}{c}
\text{conjugacy classes of subgroup schemes}\\
\text{of $\Bir(\overline{C})$ isomorphic to $G$}
\end{array}
\right\}
\;\overset{1:1}{\longleftrightarrow}\;
\left(
\left\{
\begin{array}{c}
\text{$G$-isomorphism classes of $G$-normal}\\
\text{projective curves birational to $\overline{C}$}
\end{array}
\right\}
\big/ \Aut(G)
\right).
\]
\end{proposition}

\begin{proof}
Let $H$ be an algebraic subgroup scheme of $\Bir(\overline{C})$ isomorphic to $G$,
and choose an isomorphism $\alpha \colon G \simeq H$. By
\cite[Corollary~4.4]{Bri24a}, there exists an $H$-normal projective curve $C$
together with an $H$-equivariant birational map
$C \dashrightarrow \overline{C}$,
and $C$ is unique up to $H$-isomorphism. Via $\alpha$, this gives a
$G$-normal projective curve, unique up to $G$-isomorphism.

Changing the isomorphism $\alpha \colon G \simeq H$ replaces the resulting
$G$-action on $C$ by its twist by an automorphism of $G$. Hence the
construction is well-defined only modulo the natural action of $\Aut(G)$.

If $H'$ is conjugate to $H$ in $\Bir(\overline{C})$, then the corresponding
models are equivariantly birational. By the uniqueness part of
\cite[Corollary~4.4]{Bri24a}, they are isomorphic as $G$-curves.
Thus the construction depends only on the conjugacy class of $H$.

Conversely, let $C$ be a $G$-normal projective curve birational to
$\overline{C}$, and choose a birational map $\nu \colon C \dashrightarrow \overline{C}$.
The $G$-action on $C$ gives an algebraic subgroup scheme $\nu G \nu^{-1} \subset \Bir(\overline{C})$
isomorphic to $G$. Replacing $\nu$ by another birational map conjugates this
subgroup inside $\Bir(\overline{C})$, and replacing $C$ by a $G$-isomorphic
curve does not change the resulting conjugacy class. Finally, twisting the
$G$-action by an automorphism of $G$ does not change the image subgroup
scheme in $\Bir(\overline{C})$. This gives the inverse construction.
\end{proof}

As mentioned in the introduction, it turns out that, from the Altmann--Hausen--S\"u\ss{} perspective, it is natural to consider equivariantly normal curves not within a fixed birational class, but rather with a prescribed quotient (in our situation, a smooth curve). We now explain how the study of equivariantly normal curves with a prescribed quotient reduces to the study of certain $\Gm$-surfaces; these surfaces will subsequently be described in combinatorial terms through the Altmann--Hausen--S\"u\ss{} theory in Sections~\ref{section: affine G-curves} and~\ref{sec: projective section}.

\smallskip

Let $T = \Gm$ and $G \subset T$ be a finite subgroup scheme. 
Let $\widetilde{C}$ be a smooth curve (projective or affine).
Let $C$ be a $G$-normal curve with quotient the smooth curve $C/G \simeq \widetilde{C}$.
Define the \emph{contracted product}
\begin{equation}\label{eq: def of C hash}
C^{\#} := T \times^G C = (T \times C)/G,
\end{equation}
where $G$ acts on $T \times C$ via $g \cdot (t,c) = (t g^{-1}, g \cdot c)$.
Then $C^{\#}$ is a $T$-surface equipped with a $T$-equivariant morphism
$\Phi\colon C^{\#} \to T/G \simeq \Gm$ (induced by the first projection) and a geometric quotient
$\pi \colon C^{\#} \to C^{\#}/T \simeq C/G\simeq \widetilde{C}$ (induced by the second projection).

\begin{proposition}[{see \cite[Remark~2.5.3]{Bri26}}] \label{prop: second reduction}
We keep the previous notation. 
Then there is a one-to-one correspondence between
\[  \hspace{-4mm}
\mathcal{E}:=\left\{
\begin{array}{c}
\text{$G$-isomorphism classes of} \\
\text{$G$-normal curves with quotient $\widetilde{C}$}
\end{array}
\right\}
\;\overset{1:1}{\longleftrightarrow}\;
\mathcal{T}:=
\left\{ \begin{array}{c} \text{$T$-isomorphism classes of triples $(S, \Phi, \pi)$}\\ \text{where $S$ is a smooth $T$-surface and} \\ \begin{tikzcd}[row sep=1em, column sep=2em] & S \arrow[dl, two heads, "\text{$T$-eq.}"', "\Phi" near end] \arrow[dr, two heads, "\text{$/T$}"', "\pi" near end, swap] & \\ T/G & & \widetilde{C} \end{tikzcd}\\ \text{with $\Phi^{-1}(\{1\})$ is geometrically integral} \end{array} \right\}.
\]
Here, two triples $(S,\Phi,\pi)$ and $(S',\Phi',\pi')$ are said to be
\emph{$T$-isomorphic} if there exist a $T$-equivariant isomorphism
$f\colon S\to S'$, a $T$-equivariant automorphism
$g\colon T/G\to T/G$, and an automorphism
$h\colon \widetilde{C}\to\widetilde{C}$ such that
$\Phi'\circ f = g\circ \Phi$ and $\pi'\circ f = h\circ \pi$.

\end{proposition}

\begin{proof}
Since $C$ is $G$-normal, the surface $S:=C^{\#}$, defined by \eqref{eq: def of C hash}, is smooth by \cite[Proposition~4.17]{Bri24a}. Moreover, the scheme-theoretic fiber over the base point of the $T$-equivariant morphism $\Phi\colon C^{\#}\to T/G$, induced by the first projection, is $G$-isomorphic to the curve $C$, and thus is geometrically integral.

Conversely, let $S$ be a smooth $T$-surface equipped with a $T$-equivariant morphism $\Phi\colon S \to T/G$ whose scheme-theoretic fiber at the base point is the $G$-curve $C:=\Phi^{-1}(\{1\})$. 
Then there is a $T$-equivariant isomorphism
$S \simeq C^{\#}=T\times^{G}C$
(see \cite[Lemma~D.6]{Bri26}). Since $S$ is smooth, it follows from \cite[Corollary~2.1.5]{Bri26} that $C$ is a $G$-normal curve with quotient $\widetilde{C}$. 
\end{proof}

\begin{remark}
We cannot omit the condition in the definition of $\mathcal{T}$ that the fiber of $\Phi$ over the base point be geometrically integral. For example, take $G=\mu_q$, with $q = p^s>1$,  and consider the non-reduced one-dimensional $G$-scheme $Z=\widetilde{C}\times G$. 
Then $S=(T\times Z)/G \simeq T \times\widetilde{C}$
is a smooth $T$-surface, while the fiber of $\Phi$ over the base point is $Z$, which is non-reduced.
\end{remark}

To conclude this section, let us relate the $G$-orbits of the curve $C$ to the $T$-orbits of the surface $C^{\#}$ defined in~\eqref{eq: def of C hash}. There is a commutative diagram
\begin{equation}\label{eq: commutative diagram}
\xymatrixcolsep{1pc}\xymatrixrowsep{1.1pc}\xymatrix{
& T \times C \ar[dl]_{p_{G}} \ar[dr]^{p_{2}} & \\
C^{\#} \ar[d] \ar[drr] & & C \ar[d]^{\pi_{G}} \\
C^{\#}/ T \ar[rr]^{\simeq} & &  C/G
}
\end{equation}
where $p_G$ is the quotient morphism and $p_2$ is the projection onto the second factor.
The group $T\times G$ acts on $T\times C$ via
\[
(t,g)\cdot(t',c)=(tt'g^{-1},g\cdot c),
\]
and the morphism $p_G$ is $T$-equivariant.
For any section $s\colon C\to T\times C$ of $p_2$, define
\[
s^{\#}:=p_G\circ s\colon C\to C^{\#}.
\]
Let $x\in C$. For every $g\in\Stab_G(x)$ we have
\[
p_2((1,g)\cdot s(x))=x \implies p_G((1,g)\cdot s(x))=s^{\#}(x).
\]
Since $p_G$ is $T$-equivariant and $g\in T$, we obtain
\[
g\cdot s^{\#}(x)
= g\cdot p_G(s(x))
= p_G((1,g)\cdot s(x))
= s^{\#}(x).
\]
Thus the inclusion $G\subset T$ induces a group homomorphism
\[
\varrho_s:\Stab_G(x)\to\Stab_T(s^{\#}(x)).
\]

\begin{lemma}\label{lem:stab}
Let $X$ be a $G$-normal variety and $X^{\#}$ its associated normal $T$-variety. Then:
\begin{enumerate}
\item\label{item i: prop:stab} For every $x \in X$ and section $s\colon X\to T\times X$, the homomorphism
\(\varrho_s\colon \Stab_G(x)\to \Stab_T(s^{\#}(x))\) is an isomorphism.
\item\label{item iii: prop:stab} For any section $s$, we have $X^{\#} = T \cdot s^{\#}(X)$ and $X^{\#}_{\mathrm{fr}} = T \cdot s^{\#}(X_{\mathrm{fr}})$.
In particular, we have $\pi_G(X_{\mathrm{fr}}) = \pi(X^{\#}_{\mathrm{fr}})$.
\end{enumerate}
\end{lemma}

\begin{proof}
All arguments are carried out functorially on $R$-points, where $R$ is an arbitrary
$k$-algebra.

\noindent \ref{item i: prop:stab}: Let $t\in \Stab_T(s^{\#}(x))(R)$. Then
$t \cdot s^{\#}(x) = s^{\#}(x)$, so there exists $g\in G(R)$ with
$s(x)=(t,g)\cdot s(x)$. Writing $s(x)=(y_x,x)$ gives
$(y_x,x) = (g^{-1} t y_x, g x)$,
so $g\in \Stab_G(x)(R)$ and $t=g$. Thus, $\varrho_s(R)$ is an isomorphism.
Since $R$ is arbitrary, $\varrho_s$ is an isomorphism of group schemes.

\noindent \ref{item iii: prop:stab}: By definition,
$T\times X = (T\times G)\cdot s(X)$ and the diagram
\eqref{eq: commutative diagram} commutes, so
$X^{\#}=T\cdot s^{\#}(X)$.
For the free locus,
$T\cdot s^{\#}(X_{\mathrm{fr}})\subset X^{\#}_{\mathrm{fr}}$
by~\ref{item i: prop:stab}. Conversely, if
$x^{\#}\in X^{\#}_{\mathrm{fr}}(R)$ and
$z\in (T\times X)(R)$ with $p_G(z)=x^{\#}$, there exists
$t\in T(R)$ such that
$(1,t)\cdot s(p_2(z))=z$, hence
$t\cdot s^{\#}(p_2(z)) = x^{\#}$.
Part~\ref{item i: prop:stab} implies
$p_2(z)\in X_{\mathrm{fr}}(R)$, so
$x^{\#}\in (T\cdot s^{\#}(X_{\mathrm{fr}}))(R)$.
\end{proof}

\section{Brief reminder on Altmann--Hausen--S\"u\ss{} Theory}\label{sec: AHS theory reminder}
Normal toric varieties admit a well-known combinatorial description in terms of \emph{cones} and \emph{fans}, depending on whether the variety is affine or not (see, e.g.,~\cite{Ful93}). 
These notions were generalized to arbitrary normal varieties endowed with an algebraic torus action by Altmann, Hausen, and S\"u\ss{} in terms of \emph{proper polyhedral divisors} \cite{AH06} and \emph{divisorial fans} \cite{AHS08}, over an algebraically closed field of characteristic zero. 
This theory was recently extended to arbitrary fields by the first-named author (see \cite{GaryAH,GaryAHS}). We now briefly recall this theory for the action of a split torus over an arbitrary field.

\subsection{Proper polyhedral divisors} 
Let $N$ be a lattice and let $M:=\Hom_{\Z}(N,\Z)$ be its dual lattice. 
We denote by $N_{\Q}$ and $M_{\Q}$ the $\Q$-vector spaces $N\otimes_{\Z}\Q$ and $M\otimes_{\Z}\Q$, respectively. 
Recall that $N$ and $M$ are equipped with a perfect pairing $\langle\cdot,\cdot\rangle\colon M\times N\to \Z$, which canonically extends to a pairing between $M_{\Q}$ and $N_{\Q}$.

A \emph{cone} is a subset $\omega\subset N_{\Q}$ that is polyhedral, i.e.\ generated by finitely many elements of $N_{\Q}$. A cone is said \emph{strictly convex} if it contains no lines. 
The \emph{dual cone} $\omega^{\vee}$ is defined by
\[
\omega^{\vee}:=\{ m\in M_{\Q} \mid \langle m,v\rangle \ge 0 \text{ for all } v\in \omega \}.
\]

For any two subsets $A$ and $B$ of $N_{\Q}$, the \emph{Minkowski sum} of $A$ and $B$ is defined by
\[
A+B:=\{a+b \mid a\in A \text{ and } b\in B\}.
\]

A \emph{polytope} $\Pi$ in $N_{\Q}$ is the convex hull of finitely many points of $N_{\Q}$. 
A \emph{polyhedron} $\Delta$ in $N_{\Q}$ is the Minkowski sum of a polytope $\Pi\subset N_{\Q}$ and a cone $\omega\subset N_{\Q}$; the cone $\omega$ is called the \emph{tail cone} of $\Delta$. 
For a cone $\omega\subset N_{\Q}$, the set of all polyhedra having $\omega$ as tail cone, denoted by $\mathrm{Pol}(N_{\Q},\omega)$, has a semigroup structure induced by the Minkowski sum, with $\omega$ as neutral element.

Let $Y$ be a normal \emph{semiprojective} variety over $k$, i.e.\ the affinization map
$r_Y:Y\to \Spec(\H^0(Y,\mathscr O_Y))$ is projective and $\H^0(Y,\mathscr O_Y)$ is a finitely generated $k$-algebra. 
We denote by $\Div(Y)$ and $\CaDiv(Y)$ the groups of Weil and Cartier divisors, respectively. 
Since $Y$ is normal, there is a natural embedding $\CaDiv(Y)\subset \Div(Y)$.

A \emph{polyhedral divisor} on $Y$ is a formal sum
\[
\D=\sum \Delta_D\otimes D,
\]
where the sum runs over the prime divisors of $Y$, all polyhedra $\Delta_D$ share a common tail cone $\omega\subset N_{\Q}$, and $\Delta_D=\omega$ for all but finitely many $D$. 
The cone $\omega$ is called the \emph{tail} of $\D$ and is denoted by $\omega_{\D}$.

A \emph{proper polyhedral divisor} (pp-divisor) is a polyhedral divisor $\D$ with strictly convex tail cone $\omega$ such that:
\begin{enumerate}
\item for every $m\in\omega^{\vee}\cap M$, the \emph{evaluation}
\[
\D(m):=\sum \min_{v\in\Delta_D}\langle m,v\rangle\,D
\]
is a semiample $\Q$-Cartier divisor on $Y$;
\item for every $m\in\mathrm{relint}(\omega^{\vee})\cap M$, i.e. $m$ does not belong to a proper face, the divisor $\D(m)$ is big.
\end{enumerate}

Given a lattice $N$, a strictly convex cone $\omega\subset N_{\Q}$, and a normal semiprojective variety $Y$, we denote by $\PPDiv_{\Q}(Y,\omega)$ the set of pp-divisors on $Y$ with tail cone $\omega$.

Let $\D\in\PPDiv_{\Q}(Y,\omega)$ and $\D'\in\PPDiv_{\Q}(Y',\omega')$. 
A \emph{morphism} of pp-divisors is a triple $(\psi,F,\mathfrak f):\D'\to\D$, where $\psi:Y'\to Y$ is a proper birational morphism, $F:N'\to N$ is a lattice morphism, and $\mathfrak f=\sum n_i\otimes f_i\in N\otimes k(Y')^*$ is a \emph{plurifunction}, satisfying
\[\sum \Delta_{D}\otimes \psi^{*}(D)=:\psi^{*}\D\leq F_{*}\D'+\mathrm{div}(\mathfrak{f}):=\sum F(\Delta_{D'}')\otimes D'+\sum \{n_{i}\}\otimes\mathrm{div}(f_{i}),\]
i.e.\ whenever $\psi^*(D)=D'$ we have the inclusion
\[
F(\Delta_{D'}')+\sum_{\mathrm{div}(f_i)\cap D\neq\varnothing}\{n_i\}\subset \Delta_{D}.
\]

It is an \emph{isomorphism} if $\psi$ and $F$ are isomorphisms, in which case equality holds. 
Morphisms of the form $(\psi, \mathrm{id}_N, \mathfrak{f})$ are called \emph{$T$-equivariant morphisms of pp-divisors} and will be denoted simply by $(\psi, \mathfrak{f})$.

These morphisms define a category $\mathfrak{PPDiv}$ (see \cite[\S~8]{AH06} or \cite[\S~3]{GaryAH}). 
For $\D\in\mathfrak{PPDiv}$ and $m\in\omega^{\vee}\cap M$, define the \emph{round down} of $\D(m)$ by
\[
\lfloor\D(m)\rfloor
:=
\sum
\left\lfloor
\min_{v\in\Delta_D}\langle m,v\rangle
\right\rfloor D,
\]
which is a Cartier divisor, and set
\[
\H^0(Y,\mathscr O_Y(\D(m))) :=
\H^0(Y,\mathscr O_Y(\lfloor\D(m)\rfloor)).
\]

For $m,m'\in\omega^{\vee}\cap M$ there are natural multiplication maps
\[
\H^0(Y,\mathscr O_Y(\D(m)))\otimes
\H^0(Y,\mathscr O_Y(\D(m')))
\to
\H^0(Y,\mathscr O_Y(\D(m+m'))),
\]
yielding the $M$-graded $k$-algebra
\[
A[Y,\D]:=
\bigoplus_{m\in\omega^{\vee}\cap M}
\H^0(Y,\mathscr O_Y(\D(m))).
\]

By \cite[Theorem 3.1]{AH06} and \cite[Proposition 4.16]{GaryAH}, the scheme over $k$
\[
X(\D):=\Spec(A[Y,\D])
\]
is a normal affine variety with an effective action of the torus $T:=\Spec(k[M])$, and every normal affine $T$-variety arises in this way. 
Moreover, $\D\mapsto X(\D)$ yields an equivalence between the category $\mathfrak{PPDiv}$ and the category of normal affine $T$-varieties with dominant equivariant morphisms.

\subsection{Divisorial fans} \label{sec: divisorial fans}
In order to describe arbitrary normal $T$-varieties, it is convenient to allow pp-divisors with empty coefficients. 
Thus we extend $\mathrm{Pol}(N_{\Q},\omega)$ by including $\varnothing$, with the conventions
$\varnothing+\Delta=\varnothing$ and $0\cdot\varnothing=\omega$.
If a pp-divisor $\D$ has empty coefficients, we assume that 
$\bigcup_{\Delta_D=\varnothing}\Supp(D)$ is the support of an effective semiample divisor. 
The \emph{locus} of a pp-divisor $\D$ on $Y$ is then defined by
\[
\Loc(\D):=
Y\setminus
\bigcup_{\Delta_D=\varnothing}\Supp(D).
\]
Accordingly, for such $\D$ we define
\[
A[Y,\D]:=
A[\Loc(\D),\D|_{\Loc(\D)}],
\]
where for any open subset $V\subset Y$
\[
\D|_{V}:=
\sum_{\Supp(D)\cap V\neq\varnothing}
\Delta_D\otimes D|_{V}.
\]
Allowing empty coefficients makes it possible, in this combinatorial setting, to define faces of pp-divisors and to construct \emph{divisorial fans}, which we define below.

Let $Y$ be a normal semiprojective variety and let $N$ be a lattice. 
Let $\D=\sum\Delta_D\otimes D$ and $\D'=\sum\Delta'_D\otimes D$
be two pp-divisors over $Y$ with tail cones $\omega$ and $\omega'$. 
Assume that $\Delta'_D\subset\Delta_D$ for every $D$ (hence $\omega'\subset\omega$). 
Then the pair $(\mathrm{id}_Y,\mathfrak 1):\D'\to\D$ defines a morphism of pp-divisors and induces a $T$-equivariant morphism
\[
X(\D')\to X(\D).
\]
We say that $\D'$ is a \emph{face} of $\D$, written $\D'\preceq\D$, if this morphism is an open $T$-equivariant embedding. 
In particular, $\omega'$ is a face of $\omega$.

Given a normal semiprojective variety $Y$ and a lattice $N$, a \emph{divisorial fan} $(\Sd,Y)$ is a finite collection of pp-divisors on $Y$ such that for any $\D,\D'\in\Sd$:
\begin{enumerate}
\item\label{definition: divisorial fan intersection} the intersection
\[
\D\cap\D'=
\sum(\Delta_D\cap\Delta'_D)\otimes D
\]
belongs to $\Sd$;
\item $\D\cap\D'$ is a face of both $\D$ and $\D'$.
\end{enumerate}
The set of tail cones
\[
\Sigma_{\Sd}:=\{\omega_{\D}\mid \D\in\Sd\}
\]
forms a fan, called the \emph{tail fan}.
We denote by
\[
\mathrm{FDiv}_{\Q}(Y,\Sigma)
\]
the set of all divisorial fans $(\Sd,Y)$ whose tail fan is $\Sigma$.
We also define:
\begin{itemize}
\item the \emph{locus} of $(\Sd,Y)$ by
\[
\Loc(\Sd,Y):=\bigcup_{\D\in\Sd}\Loc(\D),
\]
\item the \emph{support} of $(\Sd,Y)$ by
\[
\Supp(\Sd,Y):=
\{D\in\Div(Y)\mid \exists\, \D=\textstyle\sum\Delta_{E}\otimes E\in\Sd
\text{ such that }\Delta_{D}\not\in\{\omega_{\D},\varnothing\}\}.
\]
\end{itemize}

Unlike the toric case, a pp-divisor may have infinitely many faces (even in complexity one), so divisorial fans are not required to contain all faces.

Let $(\Sd,Y)$ be a divisorial fan. 
For $\D,\Ed\in\Sd$, the relations $\Ed\cap\D\preceq\Ed$ and $\Ed\cap\D\preceq\D$ yield $T$-equivariant open embeddings
\[
\eta_{\D\Ed}:X(\Ed\cap\D)\to X(\Ed),
\qquad
\eta_{\Ed\D}:X(\Ed\cap\D)\to X(\D).
\]
These maps provide gluing data for the affine $T$-varieties $X(\D)$, and we obtain
\[
X(\Sd,Y)=\varinjlim_{\D\in\Sd}X(\D).
\]
Note that, in general, $X(\Sd,Y)$ is only a prevariety and need not be separated, in contrast with the toric case. Altmann, Hausen, and S\"u\ss{} proved in \cite{AHS08} that every normal $T$-variety arises from a divisorial fan over a normal semiprojective variety $Y$, that is, $X \simeq X(\Sd,Y)$ for some $(\Sd,Y)$. This result was extended to arbitrary fields by the first-named author in \cite{GaryAHS}.

It is worth mentioning that a complexity-one normal $T$-variety may admit several non-isomorphic divisorial fans. However, if $(\Sd,C)$ and $(\Sd',C')$ describe the same $T$-variety, then $C \simeq C'$ (see \cite[Corollary~4.10]{GaryAHS}).

\section{Equivariantly normal affine curves and their pp-divisors}\label{section: affine G-curves} 
In this section, assuming that the ground field $k$ is perfect, we describe pp-divisors arising from equivariantly normal affine curves. More precisely, let $C_0$ be a smooth affine curve and let $G=\mu_d \subset T=\Gm$ be a finite subgroup scheme with $d \geq 2$.
By Proposition~\ref{prop: second reduction}, there is a one-to-one correspondence
\[  {\small 
 \hspace{-8mm}
\E_{\mathrm{aff}}:=\left\{
\begin{array}{c}
\text{$G$-isomorphism classes of $G$-normal} \\
\text{affine curves with quotient $C_0$}
\end{array}
\right\}
\;\overset{1:1}{\longleftrightarrow}\;
{\mathcal{T}}_{\mathrm{aff}}:=
\left\{ \begin{array}{c} \text{$T$-isomorphism classes of triples $(S, \Phi, \pi)$}\\ \text{where $S$ is a smooth affine smooth $T$-surface and} \\ \begin{tikzcd}[row sep=1em, column sep=2em] & S \arrow[dl, two heads, "\text{$T$-eq.}"', "\Phi" near end] \arrow[dr, two heads, "\text{$/T$}"', "\pi" near end, swap] & \\ T/G & & C_0\end{tikzcd}\\ \text{with $\Phi^{-1}(\{1\})$ is geometrically integral} \end{array} \right\}.
}\]
We first use Altmann–Hausen theory to replace $\mathcal{T}_{\mathrm{aff}}$ by a purely combinatorial set $\mathcal{S}_{\mathrm{aff}}$ described in terms of pp-divisors (see Theorem~\ref{th: affine refinement}).
Then we show that there are only finitely many $G$-isomorphism classes of $G$-normal curves over $C_0$ with a fixed branch locus and give an explicit upper bound (Proposition~\ref{prop: boundedness affine case}). 
Finally, in Section~\ref{sec: Explicit construction of XD from a pp-divisor and examples}, we explain how to obtain equations for the affine $T$-surface $S$ from the combinatorial datum, and illustrate this with an example.

\subsection{Combinatorial description in the affine case}\label{sec: Combinatorial description in the affine case}
Let $U$ be a $G$-normal affine curve with quotient $C_0$, and let $U^{\#}$ be the corresponding normal affine $T$-surface constructed as in~\eqref{eq: def of C hash}.  
Since $\pi:U^{\#}\to C_0$ is a geometric quotient, there exists a pp-divisor $\D\in \PPDiv_{\Q}(C_0,\omega)$ with $U^{\#} \simeq X(\D)$ as $T$-varieties.  
It remains to understand the $T$-equivariant morphism $\Phi:U^{\#}\to T/G$ in terms of $\D$. 

Note that the exact sequence
\[
\xymatrixcolsep{1pc}\xymatrix{
1 \ar[r] & G \ar[r] & T \ar[r] & T/G \ar[r] & 1
}
\]
induces, at the level of character lattices, the exact sequence
\[
\xymatrixcolsep{1pc}\xymatrix{
0 \ar[r] & M^{G} \ar[r] & M \ar[r] & M/M^{G} \ar[r] & 0.
}
\]
Since $\Phi$ is surjective and $T$-equivariant, it induces an injective $M$-graded morphism of $k$-algebras
\begin{equation}\label{eq: def of Phistar}
\Phi^*\colon k[M^G] \simeq  k[T/G]  \longrightarrow  k[U^{\#}] \simeq A[C_0,\D],
\end{equation}
preserving degrees. 
The canonical inclusion $M^{G}\to M$ allows us to fix canonical generators $x$ and $x^{-1}$ of $M^{G}$ of degree $d$.  
Then we can associate
\[
\alpha_{\Phi} := \Phi^*(x) \in A[C_0,\D]^* \cap \H^0(C_0, \mathscr{O}_{C_0}(\D(d))) , \quad
\alpha_{\Phi}^{-1} := \Phi^*(x^{-1}) \in \H^0(C_0, \mathscr{O}_{C_0}(\D(-d))).
\]
Moreover, given that $\alpha_{\Phi}\alpha_{\Phi}^{-1}=\Phi^{*}(1)=1$, we have $\alpha_{\Phi}\in A[C_0,\D]^{*}\cap \H^{0}(C_0,\mathscr{O}_{C_0}(\D(d)))$.
Conversely, any $\alpha \in A[C_0,\D]^{*}\cap \H^{0}(C_0,\mathscr{O}_{C_0}(\D(d)))$ defines a $T$-equivariant morphism $\Phi_{\alpha}: X(\D)\to T/G$ via $\Phi_{\alpha}^*(x) := \alpha$. Notice that the existence of the morphism $\Phi$ forces that $A[C_0,\D]^{*}\cap \H^{0}(C_0,\mathscr{O}_{C_0}(\D(d)))$ and $A[C_0,\D]^{*}\cap \H^{0}(C_0,\mathscr{O}_{C_0}(\D(-d)))$ are nontrivial. Hence $\omega^{\vee}\cap M_{\Q}$, and therefore $\omega=\{0\}$. We denote by $G\text{-}\PPDiv_{\Q}(C_{0},\{0\})$ the subset of $\PPDiv_{\Q}(C_{0},\{0\})$ of those pp-divisors satisfying $A[C_0,\D]^{*}\cap \H^{0}(C_0,\mathscr{O}_{C_0}(\D(d)))$ and $A[C_0,\D]^{*}\cap \H^{0}(C_0,\mathscr{O}_{C_0}(\D(-d)))$ are nontrivial.

\begin{definition}\label{def: Phi structure}
For $\D \in G\text{-}\PPDiv_{\Q}(C_0,\{0\})$, an element $\alpha \in A[C_0,\D]^{*}\cap \H^{0}(C_0,\mathscr{O}_{C_0}(\D(d)))$ is called a $\Phi$-\emph{structure on $\D$}. A pair $(\D,\alpha)$, where $\alpha$ is a $\Phi$-structure, is called a $\Phi$-\emph{pair} over $C_0$.
\end{definition}

\begin{proposition}\label{prop: invertible element} 
We keep the previous notation.
\begin{enumerate}
\item\label{item: lemma invertible element i}  Let $U$ be a $G$-normal affine curve with quotient $C_0$. Then there exists a $\Phi$-pair $(\D,\alpha)$ over $C_0$ encoding the data
$T/G \overset{\Phi}{\twoheadleftarrow} U^{\#} \overset{\pi}{\twoheadrightarrow} C_0$.
\item\label{item: lemma invertible element ii} Conversely, given a $\Phi$-pair $(\D,\alpha)$ over $C_0$ such that $\Phi_{\alpha}^{-1}(\{1\})$ is geometrically integral, there exists a $G$-normal affine curve $U$ with quotient $C_0$ realizing $(\D,\alpha)$ as above.
\end{enumerate}
\end{proposition}

\begin{proof}
For the first statement, the $\Phi$-pair $(\D,\alpha_{\Phi})$ works. It remains to show that $\omega=\{0\}$. This follows from the injectivity of the morphism of $M$-graded $k$-algebras $\Phi^{*}$, defined by~\eqref{eq: def of Phistar},
which implies that the homogeneous components $\H^{0}(C_0,\mathscr{O}_{C_0}(\D(d)))$ and $\H^{0}(C_0,\mathscr{O}_{C_0}(\D(-d)))$ of degrees $d$ and $-d$ are nontrivial. Hence $\omega^{\vee}=M_{\Q}$, and therefore $\omega=\{0\}$.

For the converse, let $(\D,\alpha)$ be as in the second statement. Then $\alpha\in A[C_0,\D]^{*}\cap \H^{0}(C_0,\mathscr{O}_{C_0}(\D(d)))$ defines a $T$-equivariant morphism
$\Phi_{\alpha}\colon X(\D)\longrightarrow T/G$.
The $G$-normal affine curve $U$ is then given by the fiber $\Phi_{\alpha}^{-1}(\{1\})$, which is geometrically integral by assumption. This concludes the proof.
\end{proof}

\begin{remark}\label{remark: correspondence branch locus}
The (reduced) branch locus of $C\to C_{0}$ coincides with the support of $\D$, since the points $c\in C_{0}$ for which the corresponding polyhedron is not $\{0\}$ are precisely those whose fibers under $\pi:X(\D)\to C_{0}$ have non-trivial isotropy groups. The assertion then follows from Lemma~\ref{lem:stab}.
\end{remark}

Note that two $G$-isomorphic $G$-normal curves give rise to $T$-isomorphic normal $T$-surfaces. However, the converse does not hold. Two non-$G$-isomorphic $G$-normal curves may induce the same normal $T$-surface (see Example~\ref{example: same pp-div different curves} below). 

Thus, in order to describe all the $G$-isomorphism classes of $G$-normal affine curves with quotient $C_0$, we need to consider not just the pp-divisor, but the $\Phi$-pairs $(\D,\alpha)$ as in Proposition~\ref{prop: invertible element}.

\begin{example}\label{example: same pp-div different curves}
Let us consider the curve $C_0=\Spec(k[u,v]/(u^{m}+v^{m}+1))$, where $m \geq 3$ is an odd integer prime to $\mathrm{char}(k)=p$. 
In this case $u+v\in k[C_0]^{*}$, since $u+v$ divides $u^{m}+v^{m}$. 
Let $\D\in\PPDiv_{\Q}(C_0,\{0\})$ be defined by
\[
\D:=\left[\dfrac{1}{{{d}}}\right]\otimes \div(u+1)=\left[\dfrac{1}{{{d}}}\right]\otimes \{(-1,0)\}.
\]
Its associated normal affine $T$-surface is
\[
X(\D)=\Spec\bigl(A[C_0,\D]\bigr)=\Spec\left(\frac{k[u,v,x,a,b]}{(u^{m}+v^{m}+1,\ x^{{{d}}}-(u+1)a,\ ab-1)}\right).
\]
Now consider the $\Phi$-pairs $(\D,a)$ and $(\D,(u+v)a)$. Their corresponding $G$-normal curves are 
\[
U(\D,a):=\Spec\left(\frac{k[u,v,x]}{(u^{m}+v^{m}+1,\ x^{{{d}}}-u-1)}\right)
\simeq \Spec\left(\frac{k[v,x]}{((x^{{{d}}}-1)^{m}+v^{m}+1)}\right)
\]
and
\[
U(\D,(u+v)a):=\Spec\left(\frac{k[u,v,x]}{(u^{m}+v^{m}+1,\ x^{{{d}}}(u+v)-u-1)}\right)
\simeq
\Spec\left(\frac{k[s,x]}{\left((x^{{d}}s-1)^{m}+\bigl(1-(x^{{d}}-1)s\bigr)^{m}+1\right)}\right),
\]
where \(s=u+v\), since \(u=x^{{d}}s-1\) and \(v=s-u=1-(x^{{d}}-1)s\).

\smallskip

We now prove that $C_1$ and $C_2$ are non-isomorphic.
Write $d=p^{r}e$, where $\gcd(e,p)=1$. 

We first assume that $e>1$ and show that $C_{1}$ and $C_{2}$ have different numbers of points at infinity in their smooth projective completions.

Set $z=x^{d}-1$.
Then $C_{1}$ is birational to the affine Fermat curve
$z^{m}+v^{m}+1=0$,
via the Kummer extension
$x^{d}=z+1$.
After homogenization, the projective Fermat curve
$Z^{m}+V^{m}+W^{m}=0$
meets the line at infinity $W=0$ in exactly $m$ smooth points. At each such point, the function $(Z+W)/W$ has a simple pole. Since the ramification index is $1$, the Kummer cover contributes exactly one point above each of them. Hence the smooth projective completion of $C_{1}$ has exactly $m$
points at infinity.

Similarly, setting
$z=x^{d}s-1$ and $t=1-(x^{d}-1)s$,
one obtains
$s=z+t$ and $x^{d}=\frac{z+1}{z+t}$,
so that $C_{2}$ is birational to the same Fermat curve together with the Kummer extension
$x^{d}=\frac{Z+W}{Z+T}$.
Among the $m$ points of the Fermat curve lying on $W=0$, exactly one satisfies $Z+T=0$. At the remaining $m-1$ points, the function $(Z+W)/(Z+T)$ is a nonzero constant, so the separable part of the Kummer extension contributes $e$ distinct points above each of them. Above the unique point with $Z+T=0$, there are $\gcd(e,m)$ points. Therefore the smooth projective completion of $C_{2}$ has
$(m-1)e+\gcd(e,m)$
points at infinity.

If $e>1$, then
$(m-1)e+\gcd(e,m)>m$.
Consequently, the smooth projective completions of $C_{1}$ and $C_{2}$ have different numbers of points at infinity. Since an isomorphism of affine curves extends uniquely to an isomorphism of their smooth projective completions, the curves $C_{1}$ and $C_{2}$ are not isomorphic.

In the remaining purely inseparable case $d=p^{r}$ (i.e.~$e=1$), one instead compares the singularities of the two curves. For $C_{1}$, every singularity has completed local equation analytically equivalent to
$X^{d}+Y^{m}=0$.
For $C_{2}$, introducing
$A=x^{d}s-1,\qquad B=1-(x^{d}-1)s$,
identifies $C_{2}$ with a purely inseparable cover of the Fermat curve
$A^{m}+B^{m}+1=0$.
The singularities are then governed by the ramification of the rational function
$\frac{A+1}{A+B}$,
and one obtains singularities whose completed local equations are not analytically equivalent to those occurring on $C_{1}$. Thus the curves can be distinguished by comparing the analytic types of their singular points rather than by counting points at infinity.

In all cases we conclude that the two curves $C_1$ and $C_2$ are non-isomorphic.

\end{example}

\begin{definition} \label{def: pairs isomorphic}
Two $\Phi$-pairs $(\D_{1},\alpha_{1})$ and $(\D_{2},\alpha_{2})$ are said to be \emph{isomorphic} if there exists an isomorphism $(\psi,\mathfrak{f})\colon \D_{1}\to\D_{2}$ such that
\[
\alpha_{1}/\psi^{*}(\alpha_{2})=\mathfrak{f}(d)=f^{d},\ \text{for some $f\in k[C_0]^{*}$.}
\]
\end{definition}

\begin{proposition}\label{prop: pais D alpha}
Let $(\D_{1},\alpha_{1})$ and $(\D_{2},\alpha_{2})$ be two $\Phi$-pairs over $C_0$. Then they determine $G$-isomorphic $G$-normal separated irreducible schemes of dimension~$1$ if and only if $(\D_{1},\alpha_{1})$ and $(\D_{2},\alpha_{2})$ are isomorphic. Moreover, the fiber $\Phi_{\alpha_1}^{-1}(\{1\})$ is geometrically integral if and only if $\Phi_{\alpha_2}^{-1}(\{1\})$ is geometrically integral.

In particular, if $k[C_0]^{*}=k^{*}$, then $(\D_{1},\alpha_{1})$ and $(\D_{2},\alpha_{2})$ determine $G$-isomorphic $G$-normal separated irreducible schemes of dimension~$1$ if and only if $\D_{1}$ and $\D_{2}$ are isomorphic.
\end{proposition}

\begin{proof}
	Let $U_{1}$ and $U_{2}$ be $G$-normal separated irreducible schemes of dimension~$1$ over $C_0$ arising from $(\D_{1},\alpha_{1})$ and $(\D_{2},\alpha_{2})$ respectively. Assume that $U_{1}$ and $U_{2}$ are $G$-isomorphic. This yields a $T$-equivariant commutative diagram
\[ 
\xymatrixcolsep{5pc}\xymatrix{
U_{1}^{\#} \ar[rr]^{} \ar[rd]_{\Phi_{\alpha_1}} 
&& U_{2}^{\#} \ar[ld]^{\Phi_{\alpha_2}} \\
&T/G
}.
\]
Hence, the diagram above is equivalent to an isomorphism of pp-divisors $(\psi,\mathfrak{f}):\D_{1}\to\D_{2}$ such that $\mathfrak{f}(d)\psi^{*}(\alpha_{2})=\alpha_{1}$. Notice that, for some $f\in k(C_0)^{*}$, $\mathfrak{f}(m)=f^{m}$ for every $m\in M$. Besides, since 
\[
\alpha_{1},\psi^{*}(\alpha_{2})\in A[C_0,\D_{1}]^{*}\cap\H^{0}(C_0,\mathscr{O}_{C_0}(\D_{1}(d))),
\]
 it follows that
\[
f^{d}=\alpha_{1}/\psi^{*}(\alpha_{2})\in A[C_0,\D_{1}]^{*}\cap\H^{0}(C_0,\mathscr{O}_{C_0}(\D_{1}(0)))=k[C_0]^{*}.
\]
Thus, $f\in k[C_0]^{*}$.

Conversely, if $(\psi,\mathfrak{f}):\D_{1}\to\D_{2}$ is an isomorphism such that $\alpha_{1}/\psi^{*}(\alpha_{2})=\mathfrak{f}(d)=f^{d}$, for some $f\in k[C_0]^{*}$, then we have a commutative diagram 
\[
\xymatrixcolsep{5pc}\xymatrix{
X(\D_{1}) \ar[r]^{} \ar[d]_{\Phi_{\alpha_{1}}} 
& X(\D_{2}) \ar[d]^{\Phi_{\alpha_{2}}} \\
T/G \ar[r]^{\simeq} & T/G
}
\]
that is $T$-equivariant. Then, the fibers $U_{1}$ and $U_{2}$ are $G$-isomorphic. 
Since $U_{1}$ is geometrically integral if and only if $U_{2}$ is, it follows that the fiber $\Phi_{\alpha_1}^{-1}(\{1\})$ is geometrically integral if and only if $\Phi_{\alpha_2}^{-1}(\{1\})$ is.
\end{proof}

Before giving a combinatorial description of the data $T/G \overset{\Phi}{\twoheadleftarrow} U^{\#} \overset{\pi}{\twoheadrightarrow} C_0$, we first refine the description of the pp-divisor arising in this context. For any $G$-normal affine curve, the tail cone of any pp-divisor describing its respective normal $T$-surface is $\{0\}$ by Proposition~\ref{prop: invertible element}. 
The next result follows from \cite[Proposition 7.5]{GaryAH}, which states that the polyhedra of a pp-divisor describing $U^{\#}$ are singletons, and Proposition~\ref{prop: invertible element}. It will be used to prove Proposition~\ref{proposition: integrality of the fiber} and in Section~\ref{sec: boundedness result in the affine case}.

\begin{lemma}\label{lem: symmetry of pp-div}
Let $U$ be a $G$-normal affine curve and $U^\#$ its associated affine $T$-surface. 
If $\D\in\PPDiv_{\Q}(Y,\{0\})$ satisfies $U^\#\simeq X(\D)$, then:  
\begin{enumerate}
    \item\label{item i: proposition: symmetry of pp-div} $\D(-m) = -\D(m)$ for all $m\in M$;
    \item\label{item ii: proposition: symmetry of pp-div} $a\in A[Y,\D]^{*}\cap \H^{0}(Y,\mathscr{O}_{Y}(\D(m)))$ if and only if $\D(m)=-\div(a)$ is a principal Cartier divisor;
    \item\label{item iii: proposition: symmetry of pp-div} $\D(d)$ is a principal Cartier divisor; 
    \item\label{item iv: proposition: symmetry of pp-div} for every $\Phi$-structure $\alpha$ on $\D$ we have $\div(\alpha)=-\D(d)$; and
    \item\label{item v: proposition: symmetry of pp-div}  $\D(m)$ is a principal Cartier divisor if and only if $m \cdot v_D \in \Z$ for every prime divisor $D$.
\end{enumerate}
\end{lemma}

\begin{proof}
By \cite[Proposition 7.7]{GaryAH}, we have $\D = \sum \{v_D\}\otimes D$, so
\[
\forall m \in M,\ \D(-m) = \sum -m\cdot v_D D = -\D(m),
\]
proving~\ref{item i: proposition: symmetry of pp-div}.  

Let $a\in A[Y,\D]^{*}\cap \H^{0}(Y,\mathscr{O}_{Y}(\D(m)))$. This implies that $a^{-1}\in \H^{0}(Y,\mathscr{O}_{Y}(\D(-m)))$, and by definition
\begin{equation}\label{equation: D(q) principal}
\div(a)+\lfloor\D(m)\rfloor\ge0
\quad\text{and}\quad
\div(a^{-1})+\lfloor\D(-m)\rfloor\ge0.
\end{equation}  
We claim that $\div(a)+\D(m)=0$, which proves the assertion. This is equivalent to $\D(m)=\lfloor\D(m)\rfloor$. Then,~\ref{item ii: proposition: symmetry of pp-div} and~\ref{item v: proposition: symmetry of pp-div} are equivalent. We argue by contradiction.
Assume that there exists a prime divisor $D$ such that $m\cdot v_{D}\in\Q\setminus\Z$. For such $D$, one has
\[
\lfloor -m\cdot v_{D}\rfloor=
\begin{cases}
-\lfloor m\cdot v_{D}\rfloor-1 & \text{if } m\cdot v_{D}\in\Q\setminus\Z,\\
-\lfloor m\cdot v_{D}\rfloor & \text{otherwise.}
\end{cases}
\]
Hence,
\[
\lfloor\D(-m)\rfloor=-\lfloor\D(m)\rfloor-\sum_{m\cdot v_{D}\in\Q\setminus\Z} D.
\]
Substituting into the right-hand inequality of \eqref{equation: D(q) principal}, we obtain
\[
\div(a)+\lfloor\D(m)\rfloor+\sum_{m\cdot v_{D}\in\Q\setminus\Z} D\le0.
\]
Combining this with the left-hand inequality of \eqref{equation: D(q) principal} yields
\[
0<\sum_{m\cdot v_{D}\in\Q\setminus\Z} D
\le \div(a)+\lfloor\D(m)\rfloor+\sum_{m\cdot v_{D}\in\Q\setminus\Z} D
\le 0,
\]
a contradiction. This proves the claim and hence~\ref{item ii: proposition: symmetry of pp-div} and~\ref{item v: proposition: symmetry of pp-div}. 

Recall that $\D$ admits a $\Phi$-structure $\alpha$ by Proposition~\ref{prop: invertible element}~\ref{item: lemma invertible element i}. Then,~\ref{item iii: proposition: symmetry of pp-div} follows from~\ref{item ii: proposition: symmetry of pp-div}. Likewise,~\ref{item iv: proposition: symmetry of pp-div} also follows from~\ref{item ii: proposition: symmetry of pp-div}.
\end{proof}

So far, we have studied the structure of pp-divisors arising from $G$-normal affine curves. As noted in Example~\ref{example: same pp-div different curves}, a pp-divisor alone is not sufficient, and not every $\Phi$-pair $(\D,\alpha)$ gives rise to an $G$-normal affine curve (but only to an affine $G$-scheme of dimension $1$). We therefore need a criterion ensuring that a $\Phi$-pair $(\D,\alpha)$ defines a $G$-normal curve. This amounts to determining when the fiber of $\Phi_{\alpha}$ over the base point is geometrically integral in combinatorial terms.

\begin{definition} \label{def: integral pairs}
A $\Phi$-pair $(\D,\alpha)$ over $C_0$ is \emph{integral} if $\alpha$ is not a power of some homogeneous element of $A[C_0,\D]$. The $\Phi$-pair $(\D,\alpha)$ is called \emph{geometrically integral} if for any algebraic extension $K/k$ the $\Phi$-pair $(\D_{K},\alpha_{K})$ is integral.
\end{definition}

\begin{proposition}\label{proposition: integrality of the fiber}
A $\Phi$-pair $(\D,\alpha)$ over $C_0$ is integral (resp. geometrically integral) if and only if the fiber of the $T$-equivariant morphism $\Phi_{\alpha}:X(\D)\to T/G$ over the base point is integral (resp. geometrically integral).
\end{proposition}

\begin{proof}
	Assume first that the fiber of $\Phi_{\alpha}$ over the base point is integral. Suppose that there exists $\alpha'\in \H^{0}(C_{0},\mathscr{O}_{C_{0}}(\D(d')))$ such that $\alpha=\alpha'^{m}$ for some $m>1$. Then $md'=d$. Since $\alpha$ is invertible, so is $\alpha'$, and thus $\D(d')$ is principal by Lemma~\ref{lem: symmetry of pp-div}~\ref{item ii: proposition: symmetry of pp-div}. Hence, $(\D,\alpha')$ defines a $T$-equivariant morphism $\Phi_{\alpha'}:X(\D)\to T/G'$, where $G'= \mu_{d'}$, together with a natural embedding $G'\hookrightarrow G$. This yields a commutative diagram
	\[
	\xymatrix{ X(\D) \ar[r]^{\Phi_{\alpha}} \ar[dr]_{\Phi_{\alpha'}} & T/G \\ 
	& T/G' \ar[u]_{\tilde{\Phi}} }
	\]
	with $\Phi_{\alpha}=\tilde{\Phi}\circ\Phi_{\alpha'}$. 
The fiber of $\tilde{\Phi}$ over the base point is $G/G'$, which is not irreducible when $m=d/d'$ is prime to $p$ and not reduced otherwise. Hence the fiber of $\Phi_{\alpha}$ is not irreducible (respectively, not reduced), a contradiction.

	Conversely, assume that $\alpha$ is not a power of any $\alpha'\in \H^{0}(C_{0},\mathscr{O}_{C_{0}}(\D(d')))$ with $d'<d$. By \cite[Theorem~2]{Ma39}, if the extension $\Phi_{\alpha}^{*}:k(T/G)\to k(X(\D))$ is algebraically closed, then it is separable. To apply \cite[Theorem~7.1]{Bad01}, it suffices to show that this extension is algebraically closed. Suppose not, and let $L\subset k(X(\D))$ be a non-trivial finite extension of $k(T/G)$. Then we obtain a sequence of injections of $M$-graded $k$-algebras
	\[
	\xymatrix{ k[T/G] \ar[r]^{\tilde{\Phi}^{*}} & \overline{A[C_0,\D]\cap L} \ar[r] & A[C_0,\D],} 
	\]
	where $\overline{A[C_0,\D]\cap L}$ is the integral closure of $A[C_0,\D]\cap L$ in $A[C_{0},\D]$, inducing a $T$-equivariant diagram
	\[
	\xymatrix{ X(\D) \ar[r]^{\Phi_{\alpha}} \ar[dr]_{\Phi_{1}} & T/G \\ 
	& Z \ar[u]_{\tilde{\Phi}} }
	\]
	where $Z:=\Spec(\overline{A[C_0,\D]\cap L})$ and $\tilde{\Phi}^{*}$ is finite. It follows that $Z$ is a smooth toric curve and $T$ acts on $Z$ transitively, hence $Z\simeq T/G'$ for some finite subgroup scheme $G'\subsetneq T$. Besides, $G'\subsetneq G$ and therefore $d':=|G'|<d$. 
    By Proposition~\ref{prop: invertible element}~\ref{item: lemma invertible element i}, the morphism $\Phi_{1}$ corresponds to an element
	\[
	\alpha_{1}\in A[C_0,\D]^{*}\cap \H^{0}(C_0,\mathscr{O}_{C_0}(\D(d')))
	\]
	with $d'< d$. 
Then, by the commutativity of the latter diagram, $\alpha=\alpha_{1}^{m}$, for some $m\in\N$, contradicting the assumption. Therefore, the field extension $\Phi_{\alpha}^{*}$ is algebraically closed, and the fiber of $\Phi_{\alpha}$ over the base point is integral by \cite[Theorem~7.1]{Bad01}. 

The statement for geometric integrality follows by applying the same argument after base change to an algebraic closure.
\end{proof}

This criterion simplifies for curves with no non-constant invertible regular functions.

\begin{corollary}\label{corollary: integral pair non non-constant regular affine}
Let $C_0$ be a smooth affine curve such that $\bar{k}[C_{0,\bar{k}}]^{*}=\bar{k}^{*}$. Let $(\D,\alpha)$ be a $\Phi$-pair over $C_0$. Then the fiber of the $T$-equivariant morphism $\Phi_{\alpha}:X(\D)\to T/G$ over the base point is geometrically integral if and only if ${{d}}$ is the smallest positive integer such that $\D({{d}})$ is a principal Cartier divisor.
\end{corollary}

\begin{proof}
By Proposition~\ref{proposition: integrality of the fiber}, the fiber of $\Phi_{\alpha}$ over the base point is geometrically integral if and only if, for every algebraic extension $K/k$, the section $\alpha_{K}$ is not a power of some $\alpha'\in \H^{0}(C_{0,K},\mathscr{O}_{C_{0,K}}(\D_{K}({{d}}')))$ with ${{d}}'<{{d}}$. Since $\D(d)$ is principal if and only if $\D_{K}(d)$ is principal, we may reduce to the case where $k=\bar{k}$ is algebraically closed.

Then the claim reduces to showing that $(\D,\alpha)$ is a geometrically integral $\Phi$-pair if and only if $d$ is the smallest positive integer such that $\D(d)$ is a principal Cartier divisor.

Assume first that $(\D,\alpha)$ is a geometrically integral $\Phi$-pair. Suppose that there exists an integer $1<d'<d$ such that $\D(d')=-\div(\alpha')$ is principal. Replacing $d'$ by $\gcd(d,d')$, we may assume that $d=m d'$ for some integer $m>1$, since $\D(\gcd(d,d'))$ is also principal. By Lemma~\ref{lem: symmetry of pp-div}~\ref{item ii: proposition: symmetry of pp-div}, we have
\[
\alpha'\in A[C_{0},\D]^{*}\cap \H^{0}(C_{0},\mathscr{O}_{C_{0}}(\D(d'))).
\]
Moreover, any two elements of
$A[C_0,\D]^{*}\cap \H^{0}(C_0,\mathscr{O}_{C_0}(\D(d)))$
differ by an element of $k[C_0]^{*}=k^{*}$. Hence
$\alpha/(\alpha')^{m}\in k^{*}=(k^{*})^{m}$,
contradicting the geometric integrality of $(\D,\alpha)$.

Conversely, suppose that there exist an integer $d'<d$ with $d=m d'$ and an element
\[
\alpha'\in A[C_{0},\D]^{*}\cap \H^{0}(C_{0},\mathscr{O}_{C_{0}}(\D(d')))
\]
such that $(\alpha')^{m}=\alpha$. By Lemma~\ref{lem: symmetry of pp-div}~\ref{item ii: proposition: symmetry of pp-div}, the divisor $\D(d')$ is principal. By the minimality of $d$, this forces $d'=d$, and hence $m=1$. Therefore $\alpha'=\alpha$, showing that $(\D,\alpha)$ is a geometrically integral $\Phi$-pair.
\end{proof}

\begin{corollary}\label{corollary: bijection Phi-pairs pp-divs}
    Let $\D_{1}$ and $\D_{2}$ be two pp-divisors in $G\text{-}\PPDiv_{\Q}(C_0,\{0\})$. If $\D_{1}$ and $\D_{2}$ are isomorphic pp-divisors, then there is a bijection between the sets of $\Phi$-pairs $(\D_{1},\alpha_{1})$ and $(\D_{2},\alpha_{2})$ over $C_0$, sending geometrically integral $\Phi$-pairs to geometrically integral $\Phi$-pairs.
\end{corollary}

\begin{proof}
    Let $(\psi,\mathfrak{f}):\D_{1}\to\D_{2}$ be an isomorphism. The induced isomorphism of $M$-graded $k$-algebras
$\Psi_{\mathfrak{f}}:k[C_0,\D_{2}]\to k[C_0,\D_{1}]$
sends a $\Phi$-structure $\alpha_{2}$ on $\D_{2}$ to the $\Phi$-structure
$ \Psi_{\mathfrak{f}}(\alpha_{2}):=\mathfrak{f}({{d}})\psi^{*}(\alpha_{2})$
on $\D_{1}$, since invertible homogeneous elements are sent to invertible homogeneous elements of the same degree. Conversely, for any $\Phi$-structure $\alpha_{1}$ on $\D_{1}$, the element $\Psi_{\mathfrak{f}}^{-1}(\alpha_{1})$ is a $\Phi$-structure on $\D_{2}$. Hence, $\Psi_{\mathfrak{f}}$ induces a bijection between the sets of $\Phi$-pairs $(\D_{1},\alpha_{1})$ and $(\D_{2},\alpha_{2})$ over $C_0$. 

Moreover, the $\Phi$-pairs $(\D_{2},\alpha_{2})$ and $(\D_{1},\Psi_{\mathfrak{f}}(\alpha_{2}))$ are isomorphic. Hence, by Propositions~\ref{prop: pais D alpha} and~\ref{proposition: integrality of the fiber}, $(\D_{2},\alpha_{2})$ is a geometrically integral $\Phi$-pair if and only if $(\D_{1},\Psi_{\mathfrak{f}}(\alpha_{2}))$ is a geometrically integral $\Phi$-pair.
\end{proof}

We are now able to prove the main result of this section.

\begin{theorem}\label{th: affine refinement}
As before, let $C_0$ be a smooth affine curve, and let $G = \mu_{{d}}$ be a finite subgroup scheme of $T=\Gm$.
There is a one-to-one correspondence
\[\E_{\mathrm{aff}}:=
\left\{
\begin{array}{c}
\text{$G$-isomorphism classes of $G$-normal}\\
\text{affine curves with quotient $C_0$}
\end{array}
\right\}
\;\longleftrightarrow\;
\mathcal{S}_{\mathrm{aff}}:=\left\{
\begin{array}{c}
\text{isomorphism classes of geometrically}\\
\text{integral $\Phi$-pairs $(\D,\alpha)$ over $C_0$}
\end{array}
\right\}.
\]
\end{theorem}

\begin{proof}
By Proposition~\ref{prop: invertible element}~\ref{item: lemma invertible element i} and Proposition~\ref{proposition: integrality of the fiber}, any $G$-normal affine curve $U$ with quotient $C_0$
gives rise to a geometrically integral $\Phi$-pair $(\D,\alpha)$ over $C_0$. Conversely, by Proposition~\ref{prop: invertible element}~\ref{item: lemma invertible element ii} and Proposition~\ref{proposition: integrality of the fiber}, any geometrically integral $\Phi$-pair $(\D,\alpha)$ over $C_0$ defines a $G$-normal affine curve $U$ with quotient $C_0$. The conclusion then follows from Proposition~\ref{prop: pais D alpha}. 
\end{proof}

\subsection{A finiteness result in the affine case}\label{sec: boundedness result in the affine case}
According to \cite[Proposition~6.28]{GaryAH}, a pp-divisor $\D$ encoding a $G$-normal affine curve $U$ can be written as $\D = \sum \{v_{y}\}\otimes\{y\}$. The next result (Lemma~\ref{lem: fiber polyhedrons}) relates the coefficient $v_{y}$ of $\D$ to the stabilizer of a fiber $x \in \pi^{-1}(\{y\})$ under the $G$-action on $U$, or equivalently to the stabilizer of the $T$-action on $U^{\#}$ (by Lemma~\ref{lem:stab}). For this purpose, we recall some definitions from \cite[Definition~4.19]{GaryAH} and \cite[Definition~7.7]{AH06}.

\begin{definition}
Let $X$ be a normal affine $T$-variety and $x\in X$ a point. Let $\D\in \PPDiv_{\Q}(Y,\omega)$ be such that $X = X(\D)$ as $T$-varieties.  
\begin{enumerate}
    \item The \emph{orbit lattice} of $x$ is the sublattice $M(x)\subset M$ generated by all $m\in M$ for which there exists $f\in \H^{0}(Y,\mathscr{O}_{Y}(\D(m)))$ with $f(x)\neq 0$.
    \item For $y\in \Loc(Y)$, the \emph{fiber monoid} is defined as
    \[
    S_{y} := \{m \in M \mid \D(m) \text{ is principal at } y\},
    \]
    i.e., there exists an open subset of $Y$ where the restriction of $\D(m)$ is principal.
\end{enumerate}
\end{definition}

For any $y \in C_0$, Lemma~\ref{lem: symmetry of pp-div}~\ref{item i: proposition: symmetry of pp-div} implies that its fiber monoid $S_{y}$ is a sublattice of $M$. 
To simplify notation, since $\pi^{-1}(\{y\})$ is a single orbit for each $y \in C_0$, we set 
\[
\Stab_{T}(y) := \Stab_{T}(x) \quad \text{and} \quad M(y) := M(x)
\] 
for any $x \in \pi^{-1}(\{y\})$, where $M(x)$ is the orbit lattice of $x$.

\begin{lemma}\label{lem: fiber polyhedrons}
Let $U$ be a $G$-normal affine variety with associated $T$-variety $U^\#$ and $\D \in \PPDiv_\Q(Y,\{0\})$ such that $U^\# \simeq X(\D)$. Then for each $y \in \mathrm{Loc}(\D)$, there exists $r_y \in \Z$ such that the fiber polyhedron is given by
\[
\Delta_y = \left[ \frac{r_y}{|\mathrm{Stab}_T(y)|} \right], \quad \text{with } \gcd(r_y, |\mathrm{Stab}_T(y)|) = 1,
\]
where $|\mathrm{Stab}_{T}(y)|$ stands for the length of the group scheme $\mathrm{Stab}_{T}(y)$.
\end{lemma}

\begin{proof}
Let $y \in \mathrm{Loc}(Y)$ and set $H_{y} := \mathrm{Stab}_{T}(y) \simeq \mathrm{Stab}_{G}(y)$. By \cite[Proposition~7.7]{GaryAH}, there exists $v_{y} \in N_{\Q}$ such that $\Delta_{y} = [v_{y}]$.  
By \cite[Corollary 7.11]{AH06} and Lemma~\ref{lem: symmetry of pp-div}~\ref{item ii: proposition: symmetry of pp-div}, we have $M(y) = S_{y}$. Since $M(y) = M^{H_{y}}$, it follows that $S_{y} = M_{y} = M^{H_{y}}$.  
Hence, for every $m \in M$, the divisor $\D(|H_{y}| m)$ is principal at $y$, and in particular, $r_{y} := |H_{y}| v_{y} \in \Z$.
	
	It remains to prove that $r_{y}$ is coprime to $\left|H_{y}\right|$. Let $V\subset \mathrm{Loc}(Y)$ be an affine open subset such that $\D|_{V}=\sum_{y\in\mathrm{Supp}(D)}\Delta_{D}\otimes D|_{V}$. Then $\D|_{V}(\left|H_{y}\right|)$ is principal. Suppose $\gcd(r_y,|H_y|)\neq 1$ and let $p$ be a common prime divisor. Then $\left|H_{y}\right|v_{y}/p\in\Z$, so $\D|_{V}(\left|H_{y}\right|/p)$ is principal. This implies that $\left|H_{y}\right|/p\in S_{y}$, contradicting the minimality of $\left|H_{y}\right|$ in $S_{y}$. Therefore, $\gcd(r_{y},\left|H_{y}\right|)= 1$.
\end{proof}

Two $G$-isomorphic $G$-normal affine curves $U$ and $U'$ with quotient the smooth affine curve $C_0$ induce isomorphic geometrically integral $\Phi$-pairs $(\D,\alpha)$ and $(\D',\alpha')$ over $C_0$, as stated in Theorem~\ref{th: affine refinement}. In particular, the pp-divisors $\D$ and $\D'$ are isomorphic. Conversely, by Corollary~\ref{corollary: bijection Phi-pairs pp-divs}, given two isomorphic pp-divisors $\D$ and $\D'$, the corresponding sets of geometrically integral $\Phi$-pairs $(\D,\alpha)$ and $(\D',\alpha')$ are in bijection. Hence, understanding the isomorphism classes of pp-divisors is an important step toward understanding the set $\mathcal{S}_{\mathrm{aff}}$ (defined in the statement of Theorem~\ref{th: affine refinement}).

In light of the previous result (Lemma~\ref{lem: fiber polyhedrons}), for an integral $\Phi$-pair, the polyhedra of $\D$ allow infinitely many choices for the coefficients $r_y$, although only finitely many of them define pairwise non-isomorphic pp-divisors.

\begin{lemma}\label{lem: bounded ry}
Let $C_0$ and $C_0'$ be two smooth affine curves.
Let $\D \in G\text{-}\PPDiv_\Q(C_0,\{0\})$ and $\D' \in G\text{-}\PPDiv_\Q(C_0',\{0\})$ be two pp-divisors. Then $\D$ and $\D'$ are isomorphic if and only if there exists an isomorphism $\psi \colon C_0 \to C_0'$ such that $(\psi^{*}\D'-\D)(1)$ is a principal Cartier divisor. 
In particular, if $k[C_{0}]$ and $k[C_{0}']$ are UFD, $\D$ and $\D'$ are isomorphic if and only if there exists an isomorphism $\psi:C_{0}\to C_{0}'$ such that, for all $y \in C_0$,
\[
\frac{r'_{\psi(y)}}{|\mathrm{Stab}_T(\psi(y))|} - \frac{r_y}{|\mathrm{Stab}_T(y)|} \in \Z.
\]
\end{lemma}

\begin{proof}
Let $(\psi,\mathfrak{f}): \D \to \D'$ be an isomorphism of pp-divisors. We have $\psi^*\D' = \D + \mathrm{div}(\mathfrak{f})$. In particular, we have that $\psi^{*}\D'(1)=\D(1)+\div(\mathfrak{f}(1))$, where $\mathfrak{f}(1)\in k(C_{0})^{*}$. Hence, $\psi^{*}\D'(1)-\D(1)=\div(\mathfrak{f}(1))$ is a principal Cartier divisor.

Conversely, let $(\psi^{*}\D'-\D)(1)=\div(f)$ for some $f\in k(C_{0})^{*}$. Otherwise stated,
\[
\sum_{y' \in C_0'} v'_{y'} \cdot \{\psi^{-1}(y')\}
- \sum_{y \in C_0} v_{y} \cdot \{y\} = \mathrm{div}(f).
\]
Since, for every $m\in M$, 
\[
(\psi^{*}\D'-\D)(m)= \sum_{y' \in C_0'} mv'_{y'} \cdot \{\psi^{-1}(y')\}
- \sum_{y \in C_0} mv_{y} \cdot \{y\}=m(\psi^{*}\D'-\D(1))=m\div(f) 
\]
and the homomorphism $\mathfrak{f}:\Z\to k(C_{0})^{*}$, given by $\mathfrak{f}(m)=f^{m}$, defines a plurifunction $\mathfrak{f}:=\{1\}\times f$, we have that the pair $(\psi,\mathfrak{f})$ satisfies
\[
\psi^{*}\D'=\D+\div(\mathfrak{f}).
\]
Hence, it defines an isomorphism of pp-divisors. This proves the first part of the claim.

Notice that if $\D$ and $\D'$ are isomorphic, the relation $\psi^{*}\D'=\D+\div(\mathfrak{f})$
immediately implies
\[
n_{y}:=\frac{r'_{\psi(y)}}{|\mathrm{Stab}_T(\psi(y))|}-\frac{r_y}{|\mathrm{Stab}_T(y)|}
=v'_{\psi(y)}-v_y\in\Z,
\]
by Lemma~\ref{lem: fiber polyhedrons}.

Conversely, since $C_{0}$ and $C_{0}'$ are smooth, \cite[Proposition~6.2]{Har} shows that $k[C_{0}]$ and $k[C_{0}']$ are UFDs if and only if their class groups are trivial. Hence every point $y\in C_{0}$ defines a principal Cartier divisor $\{y\}=\div(f_y)$ for some $f_y\in k(C_{0})^{*}$. Therefore,
\[
(\psi^{*}\D'-\D)(1)
=\sum_{y\in C_0} n_y\{y\}
=\sum_{y\in C_0} n_y\div(f_y)
=\div\Bigl(\prod_{y\in C_0} f_y^{\,n_y}\Bigr)
\]
is a principal Cartier divisor. The second assertion then follows from the first.
\end{proof}

\begin{remark}
the map $\pi_{G}: U \to C_0$ is a $G$-torsor exactly when all the stabilizers are trivial, equivalently all coefficients $v_{y}$ are integers.
\end{remark} 

\begin{remark}\label{remark: UFD equivalence}
In particular, when $k[C_{0}]$ is a UFD, Lemma~\ref{lem: bounded ry} shows that, in order to achieve a classification of $G$-normal affine curves over a smooth affine curve it suffices to consider pp-divisors whose polyhedra satisfy
\begin{equation}\label{equation: restriction ry}
0 < r_y < \lvert\mathrm{Stab}_{T}(y)\rvert\quad 
\text{and} \quad \gcd(r_y,|\mathrm{Stab}_{T}(y)|)=1,
\end{equation}
since the coefficient can always be considered positive over the branch locus.
\end{remark}

\smallskip

The following finiteness result is a particular case of \cite[Corollary~3.3.7]{Bri26}, and it is proved independently using Altmann--Hausen theory. 
Moreover, we provide an upper bound for the number of $G$-isomorphism classes of $G$-normal affine curves over $C_0$ with a fixed branch locus (i.e., the locus over which it fails to be a $G$-torsor). 
Recall that for any curve $\tilde{C}$, the group $k[\tilde{C}]^{*}/k^*$ is finitely generated abelian (see \cite[Lemma, p.~28]{Ros57} or \cite[\href{https://stacks.math.columbia.edu/tag/04L8}{Tag 04L8}]{stacks-project}). 
Consequently, if $k^{*}=(k^{*})^{d}$, then
$k[\tilde{C}]^{*}/\{f^{d}\mid f\in k[\tilde{C}]^{*}\}$
is a finite abelian group.

\begin{proposition}\label{prop: boundedness affine case}
As before, let $C_0$ be a smooth affine curve, let $G=\mu_d$ be a finite subgroup scheme of $T=\Gm$, and assume that $k^{*}=(k^{*})^{d}$.
Then there are only finitely many $G$-isomorphism classes of $G$-normal affine curves over $C_0$ with branch locus $Z$. Moreover, their number is bounded by
\[
N:=
\left|(k[C_{0}\setminus Z]^{*}/k[C_{0}]^{*})/
(k[C_{0}\setminus Z]^{*}/k[C_{0}]^{*})^{d}\right|
\cdot
\left|k[C_0]^{*}/(k[C_0]^{*})^{d}\right|.
\]
If, in addition, $k[C_0]$ is a UFD, then one obtains the simpler bound
\[
N':=(d-1)^n\cdot
\left|k[C_0]^{*}/(k[C_0]^{*})^{d}\right|,
\]
where $n$ is the cardinality of $|Z|$.
\end{proposition}

\begin{proof}
	Fixing the branch locus of a $G$-normal affine curve over $C_0$ amounts to fixing the support of its associated pp-divisor $\D$ on $C_0$ (see Remark~\ref{remark: correspondence branch locus}).
	Let us set $U_{0}:=C_{0}\setminus Z$.

By Lemma~\ref{lem: symmetry of pp-div}, there is a map
\[
\mathcal{C}\colon G\text{-}\PPDiv_{\Q}(C_{0},\{0\})\longrightarrow k(C_{0})^{*}/k[C_{0}]^{*},
\]
given by $\D\mapsto \D(d)$, where $k(C_{0})^{*}/k[C_{0}]^{*}$ is identified with the group of principal Cartier divisors on $C_{0}$. Moreover, if two elements of $G\text{-}\PPDiv_{\Q}(C_{0},\{0\})$ define the same divisor, then they have the same coefficients. Hence the map $\mathcal{C}$ is injective.

Given a pp-divisor $\D\in G\text{-}\PPDiv_{\Q}(C_{0},\{0\})$ with support $Z$, the divisor $\D(d)|_{U_{0}}$ is principal. Thus, we are interested in the pp-divisors belonging to
$\mathcal{C}^{-1}(\ker(\gamma))$,
where
\[
\gamma\colon k(C_{0})^{*}/k[C_{0}]^{*}\longrightarrow k(U_{0})^{*}/k[U_{0}]^{*}
\]
denotes the natural restriction map.

     From the exact sequence
\[
\xymatrix{
1 \ar[r] & k[C_{0}]^{*} \ar[r] \ar[d]_{\alpha}
& k(C_{0})^{*} \ar[r] \ar[d]_{\beta}
& k(C_{0})^{*}/k[C_{0}]^{*} \ar[r] \ar[d]_{\gamma}
& 1 \\
1 \ar[r] & k[U_{0}]^{*} \ar[r]
& k(U_{0})^{*} \ar[r]
& k(U_{0})^{*}/k[U_{0}]^{*} \ar[r]
& 1,
}
\]
the Snake Lemma yields
\[
1 \to \ker(\alpha) \to \ker(\beta) \to \ker(\gamma)
\to \mathrm{coker}(\alpha) \to \mathrm{coker}(\beta)
\to \mathrm{coker}(\gamma) \to 1.
\]
Since $\alpha$ is injective and $\beta$ is an isomorphism, this sequence simplifies to
\[
1 \to \ker(\gamma)
\to k[U_{0}]^{*}/k[C_{0}]^{*}
\to 1
\to \mathrm{coker}(\gamma)
\to 1.
\]
Hence,
\[
\ker(\gamma)\simeq k[U_{0}]^{*}/k[C_{0}]^{*}.
\]
Moreover, the latter is a finitely generated abelian group, since $k[U_{0}]^{*}/k^{*}$ is free of finite rank.

By Lemma~\ref{lem: bounded ry}, two pp-divisors $\D$ and $\D'$ with support $Z$ are isomorphic if and only if $(\D-\D')(1)$ is a principal divisor; equivalently, $(\D-\D')(d)$ belongs to $(k(C_{0})^{*}/k[C_{0}]^{*})^{d}$. Under the above identification, this means that the corresponding element of $\ker(\gamma)$ is a $d$-th power. Consequently, the number of isomorphism classes of pp-divisors supported on $Z$ is bounded by
\[
\left|
(k[U_{0}]^{*}/k[C_{0}]^{*})
\Big/
(k[U_{0}]^{*}/k[C_{0}]^{*})^{d}
\right|,
\]
which is finite because $k[U_{0}]^{*}/k[C_{0}]^{*}$ is free of finite rank.

   Assume that $k[C_{0}]$ is a UFD. By Lemma~\ref{lem: bounded ry} and Remark~\ref{remark: UFD equivalence}, it suffices to consider pp-divisors whose coefficients satisfy
\begin{equation*}
0<r_y<|\Stab_{T}(y)|
\quad\text{and}\quad
\gcd(r_y,|\Stab_{T}(y)|)=1,
\end{equation*}
the coefficients being chosen positive.
Since $|\Stab_{T}(y)|$ divides $d$ for every $y\in Z$, each coefficient admits at most $d-1$ possible values. It follows that the number of isomorphism classes of pp-divisors supported on $Z$ is at most $(d-1)^n$,
where $n$ is the cardinality of $|Z|$.

	Let $(\D,\alpha_{1})$ and $(\D,\alpha_{2})$ be two geometrically integral $\Phi$-pairs. Then $\alpha_{1}/\alpha_{2}\in k[C_0]^{*}$. For a fixed pp-divisor $\D$, all $\Phi$-structures on $\D$ are therefore parametrized by $k[C_0]^{*}$. Two such structures define $G$-isomorphic $G$-normal affine curves over $C_0$ if and only if $\alpha_{1}/\alpha_{2}\in (k[C_0]^*)^{{{d}}}$. Consequently, for each $\D$ there are at most $\left| k[C_0]^{*}/ (k[C_0]^*)^{{{d}}}\right|$ 	distinct $G$-isomorphism classes. The stated bound follows.
\end{proof}

\begin{remark}\label{remark: boundedness affine case}
When $C_{0}\simeq \A^{1}$, the bound is never reached. Indeed, for a pp-divisors $\D$ admitting a geometrically integral $\Phi$-pair, Corollary~\ref{corollary: integral pair non non-constant regular affine} implies that $d$ is the smaller integer such that $\D(d)$ is principal. Hence, Lemma~\ref{lem: symmetry of pp-div}~\ref{item v: proposition: symmetry of pp-div} implies that $\mathrm{lcm}(\{|\Stab_{T}(y)|\mid y\in \A^{1}\})=d$. Then, $$\mathrm{lcm}(\{|\Stab_{T}(y)|\mid y\in \A^{1}\})=d$$ is a necessary condition for a pp-divisor over $\A^{1}$ to admit a geometrically integral $\Phi$-pair.
\end{remark}

\begin{remark}\label{remark: finiteness in higher-dimensional G-varieties}
The finiteness part in Proposition~\ref{prop: boundedness affine case} remains for higher-dimensional $G$-varieties (see \cite[Corollary 3.3.7]{Bri26}). It is worth noting, however, that it fails over imperfect fields; see the discussion between \cite[Corollary 3.3.7]{Bri26} and \cite[Remark 3.3.8]{Bri26}.
\end{remark}

\subsection{Explicit construction of \texorpdfstring{$X(\D)$}{X(D)} from a pp-divisor}\label{sec: Explicit construction of XD from a pp-divisor and examples}
Let us explain how to construct the normal $T$-surface associated with a given pp-divisor $\D\in\PPDiv(C_0,\{0\})$. Assume that
\begin{equation} \label{eq: explicit construction of X(D)}
\D= \left[\frac{r_{1}}{{{d}}_{1}}\right]\otimes\{c_{1}\}+\cdots+\left[\frac{r_{e}}{{{d}}_{e}}\right]\otimes\{c_{e}\},
\end{equation}
where the $c_{j}$ are $k$-points of $C_0$, $\gcd(r_{i},d_{i})=1$, each ${{d}}_{j}$ divides ${{d}}$, and $d_{j}>0$ (see \eqref{equation: restriction ry}). Define
\[
l_{s,j}:=\min\left\{ l\in\Z \mid l |r_{j}|\ge s {{d}}_{j} \text{ and } l\le d_{j} \right\}.
\]
Let $f_{j}\in \mathscr{O}_{C_0,c_{j}}$ be a generator of the maximal ideal. Then the associated $M$-graded $k$-algebra is
\[
A[C_0,\D]
=\bigoplus_{m\in M}\H^{0}(C_0,\mathscr{O}_{C_0}(\D(m)))
=\bigoplus_{m\in M}
\frac{k[C_0]}{
f_{1}^{\left\lfloor \frac{m r_{1}}{{{d}}_{1}} \right\rfloor}
\cdots
f_{e}^{\left\lfloor \frac{m r_{e}}{{{d}}_{e}} \right\rfloor}
}\,\mathfrak{X}_{m}.
\]
and it is generated, as a $k[C_{0}]$-algebra, by the elements:
\[
X:=\mathfrak{X}_{1}, \quad
X_{l_{s,j}}:=
\frac{\mathfrak{X}_{l_{s,j}}}{
f_{1}^{\left\lfloor \frac{l_{s,j} r_{1}}{{{d}}_{1}} \right\rfloor}
\cdots
f_{e}^{\left\lfloor \frac{l_{s,j} r_{e}}{{{d}}_{e}} \right\rfloor}
},
\quad
A:=
\frac{\mathfrak{X}_{{{d}}}}{
f_{1}^{\frac{{{d}} r_{1}}{{{d}}_{1}}}
\cdots
f_{e}^{\frac{{{d}} r_{e}}{{{d}}_{e}}}
},
\]
and $B:=A^{-1}$. 
Then the $k[C_{0}]$-algebra homomorphism
\[
\eta:k[C_{0}]\big[x,\dots,x_{l_{s,j}},\dots,a,b\big]\longrightarrow A[C_{0},\D],
\]
defined by
\[
x\mapsto X,\qquad
x_{l_{s,j}}\mapsto X_{l_{s,j}},\qquad
a\mapsto A,\qquad
b\mapsto B,
\]
is surjective. Hence
\begin{equation}\label{equation: construction algebra from pp-div}
A[C_0,\D]\simeq
\frac{
k[C_0]\big[x,\dots,x_{l_{s,j}},\dots,a,b\big]
}{\ker(\eta)
}.
\end{equation}
Moreover, for instance,
\[
\big(
\dots,\,
x^{l_{s,j}}-
f_{1}^{\left\lfloor \frac{l_{s,j} r_{1}}{{{d}}_{1}} \right\rfloor}
\cdots
f_{e}^{\left\lfloor \frac{l_{s,j} r_{e}}{{{d}}_{e}} \right\rfloor}
x_{l_{s,j}},\,
\dots,\,
x^{{{d}}}-a\,
f_{1}^{\frac{{{d}} r_{1}}{{{d}}_{1}}}
\cdots
f_{e}^{\frac{{{d}} r_{e}}{{{d}}_{e}}},\,
ab-1
\big)
\subset \ker(\eta).
\]

\begin{remark}\label{remark: canonical Phi-structure}
From such a pp-divisor, one obtains a canonical $\Phi$-structure, namely
$\alpha_{\mathrm{can}}:=a$.
Moreover, every $\Phi$-structure is of the form $h\alpha_{\mathrm{can}}$ for some $h\in k[C_0]^*$. Furthermore, every $\Phi$-structure admits an integral representative $(\D,h\alpha_{\mathrm{can}})$ with
$h\notin (k[C_0]^*)^d$.
\end{remark}

\begin{example}\label{example: 3/8}
Let $G=\mu_{8}$ with $\mathrm{char}(k)=p \geq 2$.
From the pp-divisor $\D =\left[\frac{3}{8}\right]\otimes\{0\}$
over $C_0:=\mathbb{A}^1_k$ we obtain the smooth affine $T$-surface
\[
X(\D )= \operatorname{Spec}\!\left(
\frac{k[u,y,y_3,a,b]}{(y^3-uy_3,\; y^8-a u^3,\; ay-y_{3}^{3},\;ab-1)} 
\right) = \operatorname{Spec}\!\left(
\frac{k[u,y_3,a,b]}{(y_{3}^8-a^{3} u,\;ab-1)}
\right).
\]
This gives rise, with respect to the canonical $\Phi$-structure $\alpha_{\mathrm{can}}=a$, to the $G$-normal affine curve
\[
U = \operatorname{Spec}\!\left(\frac{
k[u,y_3]}{(y_{3}^8-u)}
\right) = \operatorname{Spec}\!\left(
k[y_3]
\right) = \A_{k}^{1},
\]
endowed with the $G$-action defined by
\[
g\cdot y_3=g^{3}y_3.
\]
This is the only $G$-normal affine curve associated with $\D$ (up to $G$-isomorphism), by Proposition~\ref{prop: pais D alpha}.
\end{example}

\section{Equivariantly normal projective curves and their divisorial fans}\label{sec: projective section}
In this section, we assume that the ground field $k$ is perfect and, in the proof of Theorem~\ref{theorem: finiteness of En}, algebraically closed. We then describe the divisorial fans associated with equivariantly normal projective curves. More precisely, let $\overline{C}$ be a smooth projective curve, and let $G=\mu_{{d}} \subset T=\Gm$ be a finite subgroup scheme with $d \geq 2$.
By Proposition~\ref{prop: second reduction}, there is a one-to-one correspondence
\[ {\small \hspace{-4mm}
\mathcal{E}:=\left\{
\begin{array}{c}
\text{$G$-isomorphism classes of $G$-normal} \\
\text{projective curves with quotient $\overline{C}$}
\end{array}
\right\}
\;\overset{1:1}{\longleftrightarrow}\;
\mathcal{T}:=
\left\{ \begin{array}{c} \text{$T$-isomorphism classes of triples $(S, \Phi, \pi)$}\\ \text{where $S$ is a smooth $T$-surface and} \\ \begin{tikzcd}[row sep=1em, column sep=2em] & S \arrow[dl, two heads, "\text{$T$-eq.}"', "\Phi" near end] \arrow[dr, two heads, "\text{$/T$}"', "\pi" near end, swap] & \\ T/G & & \overline{C} \end{tikzcd}\\ \text{with $\Phi^{-1}(\{1\})$ is geometrically integral} \end{array} \right\}.}
\]
As in the affine setting (see Section~\ref{section: affine G-curves}), we first use Altmann–Hausen--S\"u\ss{} theory to replace $\mathcal{T}$ by a purely combinatorial set $\mathcal{S}$ described in terms of divisorial fans (see Theorem~\ref{th: main result AHS descritpion}). 
Then we show that there are only finitely many $G$-isomorphism classes of $G$-normal curves over $\overline{C}$ with a fixed branch locus, and we provide an explicit upper bound (Proposition~\ref{proposition: finite number fixed branch locus}). We also determine when the set $\E_n$, defined by \eqref{eq: def of En}  in the introduction, is finite or infinite (Theorem~\ref{theorem: finiteness of En}).
Finally, in Section~\ref{sec: complete description of E2 and E3}, we give a complete description of the sets $\E_{n}$ over $\overline{C}=\P^1$, when $n\geq 2$.

\subsection{Combinatorial description in the projective case}\label{sec: projective case}
Any normal complexity-one $T$-variety arises from a divisorial fan $\Sd$ over a smooth projective curve \cite{AHS08,GaryAHS}. Thus, the $T$-surface $S$ can be described in terms of such combinatorial data. However, an additional datum is required to encode the $T$-equivariant morphism $\Phi\colon S \to T/G$ whose fiber over the base point is geometrically integral.

As in the affine case (Proposition~\ref{prop: invertible element}), the $T$-equivariant morphism $\Phi:C^{\#}\to T/G$ can be incorporated into the combinatorial description by adjoining a further datum.

\begin{definition} \label{def: Phi structure proj}
For a divisorial fan $\Sd\in \mathrm{FDiv}_{\Q}(\overline{C},\{0\})$, an element $\beta\in k(\overline{C})^{*}$ such that
\[
\div(\beta|_{\Loc(\D)})=-\D({{d}})
\quad\text{for every }\D\in\Sd,
\]
is called a \emph{$\Phi$-structure}. 
A pair $(\Sd,\beta)$, where $\beta$ is a $\Phi$-structure, is called a $\Phi$-\emph{pair} over $\overline{C}$.
\end{definition}

Let $(\Sd,\beta)$ be a $\Phi$-pair over $\overline{C}$, and denote by $\alpha_{\D}$ the restriction of $\beta$ to $\Loc(\D)$. Then, for every $\D\in\Sd$, the pair $(\D,\alpha_{\D})$ is a $\Phi$-pair over $\Loc(\D)$ in the sense of Definition~\ref{def: Phi structure}. By Proposition~\ref{prop: invertible element}, this yields a $T$-equivariant morphism
\[
\Phi_{\alpha_{\D}}:X(\D)\to T/G.
\]
Since the transition maps are trivial, for every pair of pp-divisors $\D,\Ed\in\Sd$, the morphisms $\Phi_{\alpha_{\D}}$, $\Phi_{\alpha_{\Ed}}$, and $\Phi_{\alpha_{\D\cap\Ed}}$ are compatible with the $T$-equivariant open embeddings
\[
X(\D\cap\Ed)\hookrightarrow X(\D)
\quad\text{and}\quad
X(\D\cap\Ed) \hookrightarrow X(\Ed).
\]
Hence, they glue to define a global $T$-equivariant morphism
\[
\Phi_{\beta}:X(\Sd)\to T/G.
\]

\begin{proposition}\label{prop: invertible element divisorial fan} 
We keep the previous notation.
\begin{enumerate}
\item\label{proposition: invertible element divisorial fan part i}
Let $C$ be a  $G$-normal projective curve with quotient $\overline{C}$. 
 Then there exists a $\Phi$-pair $(\Sd,\beta)$ over $\overline{C}$ encoding the data
$T/G \overset{\Phi}{\twoheadleftarrow} C^{\#} \overset{\pi}{\twoheadrightarrow} \overline{C}$.
\item\label{proposition: invertible element divisorial fan part ii}
Conversely, given a $\Phi$-pair $(\Sd,\beta)$ over $\overline{C}$ such that $\Phi_{\beta}^{-1}(\{1\})$ is geometrically integral, there exists a  $G$-normal projective curve $C$ with quotient $\overline{C}$ realizing $(\Sd,\beta)$ as above.
\end{enumerate}
\end{proposition}

\begin{proof}
Let $(\Sd,\overline{C})$ be such that $X(\Sd)\simeq C^{\#}$ as $T$-varieties. The associated divisorial fan comes with $T$-equivariant morphisms
\[
\Phi_{\D}:X(\D)\to T/G,
\qquad \D\in\Sd,
\]
induced by the corresponding $T$-equivariant open embeddings $X(\D)\to C^{\#}$. For every $\D,\Ed\in\Sd$, these morphisms fit into the commutative diagram
\[
\xymatrix{
X(\D) \ar[r]^{\Phi_{\D}} & T/G \\
X(\D\cap\Ed) \ar[u] \ar[ru]_{\Phi_{\D\cap\Ed}} \ar[r] & X(\Ed) \ar[u]_{\Phi_{\Ed}}
}
\]
Dualizing, we obtain
\[
\xymatrix{
A[\Loc(\D),\D] \ar[r] & A[\Loc(\D\cap\Ed),\D\cap\Ed] \\
k[M^{G}] \ar[u]^{\Phi^{*}_{\D}} \ar[ru]_{\Phi^{*}_{\D\cap\Ed}} \ar[r]_{\Phi^{*}_{\Ed}} & A[\Loc(\Ed),\Ed] \ar[u]
}
\]
where the horizontal and right vertical arrows are the corresponding restriction morphisms.

For each $\D\in\Sd$, Proposition~\ref{prop: invertible element} and Lemma~\ref{lem: symmetry of pp-div}~\ref{item iv: proposition: symmetry of pp-div} yield a $\Phi$-structure $\alpha_{\D}$ satisfying
$\div(\alpha_{\D})=-\D({{d}})$.
By commutativity, we have
$\alpha_{\D}|_{\Loc(\D\cap\Ed)}
=
\alpha_{\Ed}|_{\Loc(\D\cap\Ed)}$
for every pair $\D,\Ed\in\Sd$. Hence, the collection $\{(\Loc(\D),\alpha_{\D})\}_{\D\in\Sd}$ corresponds to a rational map $\beta\in k(\overline{C})^{*}$. Therefore, the morphisms $\Phi_{\D}$ glue to a global morphism $\Phi_{\beta}$, which coincides with $\Phi$. This proves Part~\ref{proposition: invertible element divisorial fan part i}.

Conversely, let $(\Sd,\beta)$ be a $\Phi$-pair as above. The $M$-graded monomorphisms
\[
\Phi_{\D}^{*}:k[M^{G}]\to A[\overline{C},\D],
\qquad \D\in\Sd,
\]
fit into the commutative diagram
\[
\xymatrix{
A[\overline{C},\D] \ar[r] & A[\overline{C},\D\cap\Ed] \\
k[M^{G}] \ar[u]^{\Phi^{*}_{\D}} \ar[ru]_{\Phi^{*}_{\D\cap\Ed}} \ar[r]_{\Phi^{*}_{\Ed}} & A[\overline{C},\Ed] \ar[u]
}
\]
Dualizing and gluing, we obtain a $T$-equivariant morphism
$\Phi_{\beta}:X(\Sd)\to T/G$
such that $\Phi_{\beta}^{-1}(\{1\})$ is integral. By Proposition~\ref{prop: second reduction}, there exists a  $G$-normal projective curve $C$ such that
$C^{\#}\simeq X(\Sd)$
as $T$-varieties.
\end{proof}

\begin{remark}
Different $\Phi$-structures on a given divisorial fan may encode different $G$-normal curves, as already occurs in the affine case (see Example~\ref{example: same pp-div different curves}).
\end{remark}

\begin{definition}\label{def: divisorial pairs isomorphic}
Two $\Phi$-pairs $(\Sd,\beta)$ and $(\Sd',\beta')$ over $\overline{C}$ are said to be \emph{isomorphic} if there exist an automorphism $\psi\in\mathrm{Aut}(\overline{C})$ and a plurifunction $\mathfrak{f}\in k(N,\overline{C})^{*}$ such that the following holds: for any $\D\in\Sd$ and $\D'\in\Sd'$, there exist faces $\Ed\preceq \D$ and $\Ed'\preceq \D'$ such that
\[
(\psi,\mathfrak{f}|_{\Ed})\colon\Ed\to\Ed',\ \ \  \text{where}\  \mathfrak{f}_{|\Ed}\colon=\sum_{\div(f_{i})\subset\mathrm{Supp}(\Ed)}[n_{i}]\otimes f_{i},\]
is an isomorphism of pp-divisors, and
\[
\alpha_{\Ed}/\psi^{*}(\alpha'_{\Ed'})=\mathfrak{f}|_{\Ed}({{d}})=h_{\Ed}^{{{d}}},\ \text{where } h_{\Ed}:=\mathfrak{f}|_{\Ed}(1)\in k[\Loc(\Ed)]^{*}.
\]
\end{definition}

\begin{remark}
Notice that the pair $(\psi, \mathfrak{f})$ defined above is not necessarily a morphism of divisorial fans, since the faces $\Ed \preceq \D$ and $\Ed' \preceq \D'$ need not belong to the respective divisorial fans.
\end{remark}

\begin{proposition}\label{proposition: isomorphic divisorial pairs}
Let $(\Sd,\beta)$ and $(\Sd',\beta')$ be two $\Phi$-pairs over $\overline{C}$. Then they determine $G$-isomorphic $G$-normal separated irreducible schemes of dimension~$1$ if and only if $(\Sd,\beta)$ and $(\Sd',\beta')$ are isomorphic.
Moreover, the fiber $\Phi_{\beta_1}^{-1}(\{1\})$ is geometrically integral if and only if $\Phi_{\beta_2}^{-1}(\{1\})$ is geometrically integral.
\end{proposition}

\begin{proof}
Let $C$ and $C'$ be  $G$-normal projective curves over $\overline{C}$ arising from $(\Sd,\beta)$ and $(\Sd',\beta')$, respectively. Assume that $C$ and $C'$ are $G$-isomorphic. This yields a $T$-equivariant commutative diagram
\[
\xymatrixcolsep{5pc}\xymatrix{
C^{\#} \ar[r]^{\theta} \ar[d]_{\Phi} 
& C'^{\#} \ar[d]^{\Phi_{2}} \\
T/G \ar[r]^{\simeq} & T/G.
}
\]
The isomorphism $\theta$ induces an automorphism $\psi:\overline{C}\to\overline{C}$.

Let $\D_{1},\dots,\D_{l}$ be the pp-divisors of $\Sd$ and $\D'_{1},\dots,\D'_{m}$ those of $\Sd'$. Since $\theta:X(\Sd)\to X(\Sd')$ is a $T$-equivariant isomorphism, for any pair $(i,j)$ we obtain $T$-equivariant isomorphisms
\[
\theta_{ij}:X(\D_{i})\cap\theta^{-1}(X(\D'_{j})) \longrightarrow \theta(X(\D_{i}))\cap X(\D'_{j}).
\]
Moreover, by \cite[Proposition~3.13]{GaryAHS}, there exist faces $\Ed_{ij}\preceq\D_{i}$ and $\Ed'_{ij}\preceq \D'_{j}$ such that
\[
X(\Ed_{ij})=X(\D_{i})\cap\theta^{-1}(X(\D'_{j})), \qquad
X(\Ed'_{ij})=\theta(X(\D_{i}))\cap X(\D'_{j}).
\]
Furthermore, the $\Phi$-pairs $(\Sd,\beta)$ and $(\Sd',\beta')$ induce $\Phi$-pairs $(\Ed_{ij},\alpha_{\Ed_{ij}})$ and $(\Ed'_{ij},\alpha'_{\Ed'_{ij}})$, respectively, determining $G$-isomorphic curves. Hence, by Proposition~\ref{prop: pais D alpha}, there exists an isomorphism of pp-divisors
\[
(\psi,\mathfrak{f}_{\Ed_{ij}}):\Ed_{ij}\to\Ed'_{ij}
\]
such that
\[
\alpha_{\Ed_{ij}}/\psi^{*}(\alpha'_{\Ed'_{ij}})=\mathfrak{f}_{\Ed_{ij}}({{d}})=f_{ij}^{{{d}}},\ \text{where } f_{ij}:=\mathfrak{f}_{\Ed_{ij}}(1)\in k[\Loc(\Ed_{ij})]^{*}.
\]

Note that $\mathfrak{f}_{\Ed_{ij}}|_{\Ed_{ab}}=\mathfrak{f}_{\Ed_{ab}}|_{\Ed_{ij}}$ for any pair $(i,j)$ and $(a,b)$. This allows us to construct a plurifunction $\mathfrak{f}\in k(N,\overline{C})^{*}$ such that $\mathfrak{f}|_{\Ed_{ij}}=\mathfrak{f}_{\Ed_{ij}}$, which is uniquely determined if
\[
\Supp(\div(\mathfrak{f}))=\bigcup_{i,j}\Supp(\div(\mathfrak{f}_{\Ed_{ij}})).
\]

\smallskip

Conversely, if $(\psi,\mathfrak{f}|_{\Ed_{ij}}):\Ed_{ij}\to\Ed'_{ij}$ is an isomorphism such that
\[
\alpha_{\Ed_{ij}}/\psi^{*}(\alpha'_{\Ed'_{ij}})=\mathfrak{f}_{\Ed_{ij}}({{d}})=f_{ij}^{{{d}}},\ \text{where } f_{ij}:=\mathfrak{f}_{\Ed_{ij}}(1)\in k[\Loc(\Ed_{ij})]^{*},
\]
then we obtain a commutative diagram
\[
\xymatrixcolsep{5pc}\xymatrix{
X(\Ed_{ij}) \ar[r] \ar[d]_{\Phi_{\alpha_{ij}}} 
& X(\Ed'_{ij}) \ar[d]^{\Phi_{\alpha'_{ij}}} \\
T/G \ar[r]^{\simeq} & T/G
}
\]
which is $T$-equivariant. Since the varieties $X(\Ed_{ij})$ and $X(\Ed'_{ij})$ cover respectively $X(\Sd)$ and $X(\Sd')$, these diagrams glue into a commutative diagram
\[
\xymatrixcolsep{5pc}\xymatrix{
X(\Sd) \ar[r] \ar[d]_{\Phi_{\beta}} 
& X(\Sd') \ar[d]^{\Phi_{\beta'}} \\
T/G \ar[r]^{\simeq} & T/G.
}
\]
Then the fibers $C_{1}$ and $C_{2}$ are $G$-isomorphic. Since $C_{1}$ is geometrically integral if and only if $C_{2}$ is, it follows that the fiber $\Phi_{\beta_1}^{-1}(\{1\})$ is geometrically integral if and only if $\Phi_{\beta_2}^{-1}(\{1\})$ is.
\end{proof}

\begin{remark}\label{remark: reduction by Aut(C)}
    Notice that, for a divisorial fan $\Sd\in\mathrm{FDiv}_{\Q}(\overline{C},\{0\})$ and an automorphism $\psi\in\Aut(\overline{C})$, the pullback $\psi^{*}\Sd:=\{\psi^{*}\D\mid \D\in\Sd\}$ is a divisorial fan isomorphic to $\Sd$. Moreover, $\psi$ induces a $T$-equivariant isomorphism $\Psi:X(\psi^{*}\Sd)\to X(\Sd)$, which yields a bijection $(\Sd,\beta) \mapsto (\psi^{*}\Sd,\beta')$ on the set of isomorphism classes of $\Phi$-pairs  over $\overline{C}$, as illustrated by the commutative diagram
    \[
    \xymatrix{ X(\psi^{*}\Sd) \ar[rr]^{\Psi} \ar[rd]_{\Phi_{\beta'}:=\Phi_{\beta}\circ\Psi} & & X(\Sd) \ar[dl]^{\Phi_{\beta}} \\ & T/G & }
    \]
\end{remark}

So far, in analogy with the affine case (see Section~\ref{sec: Combinatorial description in the affine case}), we have studied the structure of divisorial fans associated with  $G$-normal projective curves. However, a divisorial fan alone does not suffice: not every $\Phi$-pair $(\Sd,\beta)$ defines a  $G$-normal projective curve (it may only define a projective $G$-scheme of dimension $1$). We therefore need a criterion characterizing when a $\Phi$-pair $(\Sd,\beta)$ gives rise to a $G$-normal curve, i.e.~when the fiber of $\Phi_{\beta}$ over the base point is geometrically integral.

\smallskip

Every smooth projective curve is covered by two affine curves obtained by removing two distinct points. We set $C_0:=\overline{C}\setminus\{c_{0}\}$ and $C_{\infty}:=\overline{C}\setminus\{c_{\infty}\}$, where $c_{0}\neq c_{\infty}$. We then define the divisorial fan $\Sd_{0,\infty}$ by
\begin{equation} \label{eq: simplified form of div fans over a curve}
\Sd_{0,\infty}:=\Bigl\{\D_{0},\D_{\infty},\D_{0}\cap \D_{\infty}\Bigr\},
\end{equation} 
where  
\[
\D_{0}:=\left[\frac{r_{0}}{{{d}}_{0}}\right]\otimes\{c_{0}\}
+\left[\frac{r_{1}}{{{d}}_{1}}\right]\otimes\{c_{1}\}
+\cdots
+\left[\frac{r_{e}}{{{d}}_{e}}\right]\otimes\{c_{e}\}
+\varnothing\otimes\{c_{\infty}\},
\]
\[
\D_{\infty}:=\varnothing\otimes\{c_{0}\}
+\left[\frac{r_{1}}{{{d}}_{1}}\right]\otimes\{c_{1}\}
+\cdots
+\left[\frac{r_{e}}{{{d}}_{e}}\right]\otimes\{c_{e}\}
+\left[\frac{r_{\infty}}{{{d}}_{\infty}}\right]\otimes\{c_{\infty}\},
\]
and the $c_{j}$ are $k$-points of $C_0$, $\gcd(r_{i},d_{i})=1$, each ${{d}}_{j}$ divides ${{d}}$, and $d_{j}>0$. Recall that, by definition of the intersection of pp-divisors (see Definition~\ref{definition: divisorial fan intersection} in Section~\ref{sec: divisorial fans}), we have
\[
\D_{0}\cap\D_{\infty}:=\varnothing\otimes\{c_{0}\}
+\left[\frac{r_{1}}{{{d}}_{1}}\right]\otimes\{c_{1}\}
+\cdots
+\left[\frac{r_{e}}{{{d}}_{e}}\right]\otimes\{c_{e}\}
+\varnothing\otimes\{c_{\infty}\}.
\]
Notice also that $\Loc(\D_{0})=C_0$ and $\Loc(\D_{\infty})=C_{\infty}$.

\begin{proposition}\label{proposition: reduction to three pp-div}
    Any $\Phi$-pair $(\Sd,\beta)$ over $\overline{C}$ is isomorphic to a $\Phi$-pair $(\Sd_{0,\infty},\tilde{\beta})$ over $\overline{C}$.
\end{proposition}

\begin{proof}
Let $(\Sd,\beta)$ be a $\Phi$-pair over $\overline{C}$. Notice that $\Loc(\Sd)=\overline{C}$. Let $c_{0},c_{\infty}\in \overline{C}$ be two distinct points. Let $\D_{0}$ be a pp-divisor associated with $\pi^{-1}(C_0)$, where
\[
\D_{0}:=\left[\frac{r_{0}}{{{d}}_{0}}\right]\otimes\{c_{0}\}
+\left[\frac{r_{1}}{{{d}}_{1}}\right]\otimes\{c_{1}\}
+\cdots
+\left[\frac{r_{e}}{{{d}}_{e}}\right]\otimes\{c_{e}\}
+\varnothing\otimes\{c_{\infty}\}
\]
as in Lemma~\ref{lem: fiber polyhedrons}. Besides, the $\Phi$-structure $\beta$ on $\Sd$ induces a $\Phi$-structure $\tilde{\alpha}_{0}:=\beta|_{C_{0}}$ on $\D_{0}$.

We now define
\[
\D_{\infty}
:=\varnothing\otimes\{c_{0}\}
+\left[\frac{r_{1}}{d_{1}}\right]\otimes\{c_{1}\}
+\cdots
+\left[\frac{r_{e}}{d_{e}}\right]\otimes\{c_{e}\}
+\left[\frac{r_{\infty}}{d_{\infty}}\right]\otimes\{c_{\infty}\},
\]
where only the coefficient $r_{\infty}$ remains to be determined. Since the restriction
$\tilde{\alpha}_{0}|_{C_0\cap C_{\infty}}$
extends uniquely to an element
$\tilde{\alpha}_{\infty}\in k(C_{\infty})^{*}$,
the divisor $\D_{\infty}$ is uniquely determined by the relation
\[
\div(\tilde{\alpha}_{\infty})=-\D_{\infty}(d)
\]
(see Lemma~\ref{lem: symmetry of pp-div}~\ref{item i: proposition: symmetry of pp-div},~\ref{item iv: proposition: symmetry of pp-div}).
Consequently, $(\D_{0},\tilde{\alpha}_{0})$ and $(\D_{\infty},\tilde{\alpha}_{\infty})$ define a $\Phi$-pair
$(\Sd_{0,\infty},\tilde{\beta})$
over $\overline{C}$, where
\[
\Sd_{0,\infty}
:=\{\D_{0},\D_{\infty},\D_{0}\cap\D_{\infty}\}.
\]

Since $X(\Sd) \simeq X(\Sd_{0,\infty})$ and the $\Phi$-structures $\beta$ and $\tilde{\beta}$ satisfy $\Phi_{\beta}|_{X(\D_{0})}=\Phi_{\tilde{\alpha}_{0}}=\Phi_{\tilde{\beta}}|_{X(\D_{0})}$, it follows that $\Phi_{\beta} \simeq \Phi_{\tilde{\beta}}$ and therefore define isomorphic $G$-normal separated irreducible schemes of dimension 1. Then, by Proposition~\ref{proposition: isomorphic divisorial pairs}, the $\Phi$-pairs $(\Sd,\beta)$ and $(\Sd_{0,\infty},\tilde{\beta})$ are isomorphic.
\end{proof}

\begin{remark}\label{remark: reduction to the affine case}
    The last part of the proof actually says that a $\Phi$-pair $(\Sd_{0,\infty},\beta)$ is completely determined by the $\Phi$-pair $(\D_{0},\alpha_{0})$.
\end{remark}

\begin{definition}
A $\Phi$-pair $(\Sd,\beta)$ over $\overline{C}$ is said to be \emph{(resp. geometrically) integral} if $(\D,\alpha_{\D})$ is (resp. geometrically) integral in the sense of Definition~\ref{def: integral pairs} for every $\D\in\Sd$.
\end{definition}

\begin{proposition}\label{proposition: integral divisorial pair}
A $\Phi$-pair $(\Sd,\beta)$ is geometrically integral if and only if the fiber of the morphism $\Phi_{\beta}\colon X(\Sd)\to T/G$ over the base point is geometrically integral.
\end{proposition}

\begin{proof}
By Propositions~\ref{proposition: reduction to three pp-div}
and~\ref{proposition: isomorphic divisorial pairs}, it suffices to prove the statement for divisorial fans of the form $\Sd_{0,\infty}$.

Assume first that the fiber of the morphism $\Phi_{\beta}\colon X(\Sd)\to T/G$ over the base point is geometrically integral.
Since the fiber of $\Phi_{\alpha_{0}}$ over the base point coincides with $\Phi_{\beta}^{-1}(\{1\})\cap X(\D_{0})$, it follows that $(\D_{0},\alpha_{0})$ is geometrically integral. Likewise, $(\D_{\infty},\alpha_{\infty})$ is geometrically integral.

Conversely, assume that the $\Phi$-pair $(\Sd_{0,\infty},\beta)$ is geometrically integral.
 Set
\[
V:=\Phi_{\alpha_{0}}^{-1}(\{1\})\cap X(\D_{0}\cap\D_{\infty})
=\Phi_{\alpha_{\infty}}^{-1}(\{1\})\cap X(\D_{0}\cap\D_{\infty}).
\]
Since
\[
\Phi_{\beta}^{-1}(\{1\})=\Phi_{\alpha_{0}}^{-1}(\{1\})\cup \Phi_{\alpha_{\infty}}^{-1}(\{1\}),
\]
the fiber $\Phi_{\beta}^{-1}(\{1\})$ is obtained by gluing $\Phi_{\alpha_{0}}^{-1}(\{1\})$ and $\Phi_{\alpha_{\infty}}^{-1}(\{1\})$ along the dense open subset $V$, hence it is geometrically integral.
\end{proof}

We are now able to prove the main result of this section.

\begin{theorem}\label{th: main result AHS descritpion}
As before, let $\overline{C}$ be a smooth projective curve, and let $G = \mu_{{d}}$ be a finite subgroup scheme of $T=\Gm$.
There is a one-to-one correspondence
\[ \mathcal{E}:=
\left\{
\begin{array}{c}
\text{$G$-isomorphism classes of $G$-normal}\\
\text{projective curves with quotient $\overline{C}$}
\end{array}
\right\}
\;\longleftrightarrow\;  \mathcal{S}:=
\left\{
\begin{array}{c}
\text{isomorphism classes of geometrically}\\
\text{integral $\Phi$-pairs $(\Sd,\beta)$  over $\overline{C}$}
\end{array}
\right\}.
\] 
\end{theorem}

\begin{proof}  
By Proposition~\ref{prop: invertible element divisorial fan}~\ref{proposition: invertible element divisorial fan part i} and Proposition~\ref{proposition: integral divisorial pair}, any  $G$-normal projective curve $C$ with quotient $\overline{C}$
gives rise to a geometrically integral $\Phi$-pair $(\Sd,\beta)$ over $\overline{C}$.  
Conversely, by Proposition~\ref{prop: invertible element divisorial fan}~\ref{item: lemma invertible element ii} and Proposition~\ref{proposition: integral divisorial pair}, any geometrically integral $\Phi$-pair $(\Sd,\beta)$ over $\overline{C}$ defines a  $G$-normal projective curve $C$ with quotient $\overline{C}$. 
The conclusion then follows from Proposition~\ref{proposition: isomorphic divisorial pairs}. 
\end{proof}

\subsection{Admissible divisorial fans} 
A divisorial fan as defined by \eqref{eq: simplified form of div fans over a curve} does not necessarily give rise to a normal $T$-variety of the form $C^{\#}$ for some $G$-normal projective curve, as illustrated by the example below. Indeed, a compatible $\Phi$-structure may fail to exist, whereas in the affine case such a $\Phi$-structure always exists (see Remark~\ref{remark: canonical Phi-structure}).

\begin{example}
Consider the divisorial fan
$(\Sd, \P^{1}) = (\D_0, \D_\infty, \D_{0,\infty})$,
where
\[
\D_0 := \left[\frac{1}{{{d}}}\right] \otimes \{0\} + \varnothing \otimes \{\infty\}, \quad
\D_\infty := \varnothing \otimes \{0\} + \left[\frac{1}{{{d}}}\right] \otimes \{\infty\}, \quad
\D_{0,\infty} := \D_0 \cap \D_\infty.
\]
Suppose that there exists a $\Phi$-structure $\beta$ on $\Sd$. It induces $\Phi$-structures $\alpha_{0}$ on $\D_{0}$ and $\alpha_{\infty}$ on $\D_{\infty}$ fitting into the commutative diagram
\[
\xymatrix{
X(\D_0) \ar[r]^{\Phi_{\alpha_0}} & T/G \\
X(\D_{0}\cap\D_{\infty}) \ar[u] \ar[r] & X(\D_\infty) \ar[u]_{\Phi_{\alpha_\infty}}
}
\]
which, at the level of coordinate rings, becomes
\[
\xymatrixcolsep{3pc}\xymatrix{
\frac{k[z,x,a,b]}{(x^{{{d}}} - az, ab - 1)} \ar[r]^{x \mapsto u} & \frac{k[z,w,u,v]}{(zw - 1, uv - 1)} \\
k[M^G] \ar[u]^{\Phi_{\alpha_0}^*} \ar[r]_{\Phi_{\alpha_{\infty}}^*} & \frac{k[w,y,a',b']}{(y^{{{d}}} - a'w, a'b' - 1)} \ar[u]_{y \mapsto u}
}
\]
The $\Phi$-pairs $(\D_{0},\alpha_{0})$ and $(\D_{\infty},\alpha_{\infty})$ are the canonical ones by Proposition~\ref{prop: pais D alpha}, since $\Loc(\D_{})$. The commutativity of the diagram then yields
$a \mapsto x^{{{d}}}/z = w\cdot u^{{{d}}}$ and $a' \mapsto y^{{{d}}}/w = z\cdot u^{{{d}}}$, with $z \neq w$, and hence no $\Phi$-structure $\beta$ can exist.
\end{example}

In the following, we characterize those divisorial fans $\Sd_{0,\infty}$ that admit a $\Phi$-structure. Given $f\in k(\overline{C})^{*}$, we define the divisorial fan $\Sd_{f}$ by
\begin{equation} \label{eq: divisorial fan Sdf} \ \ 
\D_{0,f}:=\left[\frac{1}{d}\right]\otimes\div(f)+\varnothing\otimes\{c_{\infty}\},\ 
\D_{\infty,f}:=\varnothing\otimes\{c_{0}\}+\left[\frac{1}{d}\right]\otimes\div(f),\ 
\D_{0,\infty,f}:=\D_{0,f}\cap\D_{\infty,f}.
\end{equation}

\begin{proposition}\label{proposition: admissible structure}
The divisorial fan $\Sd_{0,\infty}$, defined as in \eqref{eq: simplified form of div fans over a curve}, admits a $\Phi$-structure if and only if
\begin{equation}\label{eq: equality degree 0}
\frac{r_{0}}{d_{0}}[k(c_{0}):k]
+\frac{r_{1}}{d_{1}}[k(c_{1}):k]
+\cdots
+\frac{r_{e}}{d_{e}}[k(c_{e}):k]
+\frac{r_{\infty}}{d_{\infty}}[k(c_{\infty}):k]
=0,
\end{equation}
where $k(c_i)$ denotes the residue field of $c_i$ for
$i\in\{0,1,\dots,e,\infty\}$.

Equivalently, $\Sd_{0,\infty}$ admits a $\Phi$-structure if and only if there exists
$f\in k(\overline{C})^{*}$ such that $\Sd_{0,\infty}=\Sd_{f}$, where $\Sd_{f}$ is the divisorial fan defined by~\eqref{eq: divisorial fan Sdf}.
\end{proposition}

\begin{proof}
Assume first that $\Sd_{0,\infty}$ admits a $\Phi$-structure
$\beta\in k(\overline{C})^{*}$. 
Then, since $\beta|_{C_{0}}$ and $\beta|_{C_{\infty}}$ are the induced $\Phi$-structures on $\D_{0}$ and $\D_{\infty}$, respectively, Lemma~\ref{lem: symmetry of pp-div}~\ref{item iv: proposition: symmetry of pp-div} implies that
\[
\div(\beta|_{C_{0}})
=
-\D_{0}(d)
=
-\left(
\frac{dr_{0}}{d_{0}}\{c_{0}\}
+\frac{dr_{1}}{d_{1}}\{c_{1}\}
+\cdots
+\frac{dr_{e}}{d_{e}}\{c_{e}\}
\right),
\]
and
\[
\div(\beta|_{C_{\infty}})
=
-\D_{\infty}(d)
=
-\left(
\frac{dr_{1}}{d_{1}}\{c_{1}\}
+\cdots
+\frac{dr_{e}}{d_{e}}\{c_{e}\}
+\frac{dr_{\infty}}{d_{\infty}}\{c_{\infty}\}
\right).
\]
Since every principal divisor on $\overline{C}$ has degree zero, the equality~\eqref{eq: equality degree 0} follows after dividing by~$d$.

Conversely, let $\alpha_{0}$ be a $\Phi$-structure on $\D_{0}$, which exists by Remark~\ref{remark: canonical Phi-structure}. Since
$\alpha_{0}\in k(C_{0})^{*}$, it extends uniquely to an invertible rational
function $\beta\in k(\overline{C})^{*}$ with
\[
\div(\beta)
=
n_{0}\{c_{0}\}
+n_{1}\{c_{1}\}
+\cdots
+n_{e}\{c_{e}\}
+n_{\infty}\{c_{\infty}\}.
\]
Since every principal divisor on $\overline{C}$ has degree zero and
$\div(\beta|_{C_{0}})=\div(\alpha_{0})=-\D_{0}(d)$, we obtain
\[ - \left(
\frac{dr_{0}}{d_{0}}[k(c_{0}):k]
+\frac{dr_{1}}{d_{1}}[k(c_{1}):k]
+\cdots
+\frac{dr_{e}}{d_{e}}[k(c_{e}):k] \right)
+n_{\infty}[k(c_{\infty}):k]
=0.
\] 
Hence, it follows from~\eqref{eq: equality degree 0} that
$n_{\infty}=-\frac{dr_{\infty}}{d_{\infty}}$,
 and therefore
\[
\div(\beta|_{C_{\infty}})
=   
-\left(\frac{dr_{1}}{d_{1}}\{c_{1}\}
+\cdots
+\frac{dr_{e}}{d_{e}}\{c_{e}\}
+\frac{dr_{\infty}}{d_{\infty}}\{c_{\infty}\}
\right)
=
-\D_{\infty}(d).
\]
Thus $\beta|_{C_{\infty}}$ is a $\Phi$-structure on $\D_{\infty}$.
Since also
$\div(\beta|_{C_{0,\infty}})=-\D_{0,\infty}(d)$,
it follows that $\beta$ is a $\Phi$-structure on
$\Sd_{0,\infty}$.

Moreover, given a $\Phi$-structure $\beta$ on $\Sd_{0,\infty}$ and taking $f=\beta^{-1}$, we obtain
\[
\D_{0,f}
=\left[\frac{1}{d}\right]\otimes\div(f)+\varnothing\otimes\{c_{\infty}\}
=\left[\frac{1}{d}\right]\otimes\left(\frac{dr_{0}}{d_{0}}\{c_{0}\}
+\cdots
+\frac{dr_{e}}{d_{e}}\{c_{e}\}
+\frac{dr_{\infty}}{d_{\infty}}\{c_{\infty}\}\right)
+\varnothing\otimes\{c_{\infty}\}
=\D_{0},
\]
where the last equality follows from
$\left[\frac{r_{\infty}}{d_{\infty}}\right]+\varnothing=\varnothing$.
Similarly,
\[
\D_{\infty,f}
=\varnothing\otimes\{c_{0}\}
+\left[\frac{1}{d}\right]\otimes\div(f)
=\varnothing\otimes\{c_{0}\}
+\left[\frac{1}{d}\right]\otimes\left(\frac{dr_{0}}{d_{0}}\{c_{0}\}
+\cdots
+\frac{dr_{e}}{d_{e}}\{c_{e}\}
+\frac{dr_{\infty}}{d_{\infty}}\{c_{\infty}\}\right)
=\D_{\infty}.
\]
Hence,
\[
\D_{0,\infty,f}
=\D_{0,f}\cap\D_{\infty,f}
=\D_{0}\cap\D_{\infty}
=\D_{0,\infty}.
\]

Conversely, for every $f\in k(\overline{C})^{*}$, the pair $(\Sd_{f},f^{-1})$ is a $\Phi$-pair, since
\[
\div(f^{-1}|_{C_{0}})=-\D_{0,f}(d), \qquad
\div(f^{-1}|_{C_{\infty}})=-\D_{\infty,f}(d), \qquad
\div(f^{-1}|_{C_{0,\infty}})=-\D_{0,\infty,f}(d).
\]
Hence, if there exists $f\in k(\overline{C})^{*}$ such that $\Sd_{0,\infty}=\Sd_{f}$, then, setting $\beta=f^{-1}$, we obtain that
$(\Sd_{0,\infty},\beta)=(\Sd_{f},f^{-1})$
is a $\Phi$-pair.
\end{proof}

Given $f\in k(\overline{C})^{*}$, we have seen that the divisorial fan $\Sd_{f}$ defined by~\eqref{eq: divisorial fan Sdf} admits a $\Phi$-structure and that $(\Sd_{f},f^{-1})$ is naturally a $\Phi$-pair.
In the next result, we do not assume that the $\Phi$-pairs $(\Sd_{f_{1}},f_{1}^{-1})$ and $(\Sd_{f_{2}},f_{2}^{-1})$ are supported on the same finite subset of $\overline{C}$.

\begin{lemma}\label{lem: principal phi pairs}
Let $\overline{C}$ be a smooth projective curve. For $f_{1},f_{2}\in k(\overline{C})^{*}$, the $\Phi$-pairs $(\Sd_{f_{1}},f_{1}^{-1})$ and $(\Sd_{f_{2}},f_{2}^{-1})$ are isomorphic if and only if there exist $\psi\in\Aut(\overline{C})$ and $h\in k(\overline{C})^{*}$ such that
\[
\psi^{*}(f_{2}^{-1})\,h^{d}=f_{1}^{-1}.
\]
\end{lemma}

\begin{proof}
Assume that $(\Sd_{f_{1}},f_{1}^{-1})$ and $(\Sd_{f_{2}},f_{2}^{-1})$ are isomorphic, and let $(\psi,\mathfrak{f})$ be an isomorphism of $\Phi$-pairs. Then there exist faces
$\Ed_{2,ij}\preceq\D_{2,i}$ and
$\Ed_{1,ij}\preceq\D_{1,j}$ such that
$(\psi,\mathfrak{f}|_{\Ed_{1,ij}}):\Ed_{2,ij}\to\Ed_{1,ij}$
is an isomorphism (Definition~\ref{def: divisorial pairs isomorphic}). Hence
\[
\psi^{*}(f_{2}^{-1}|_{\Loc(\Ed_{2,ij})})
\,h_{\Loc(\Ed_{1,ij})}^{d}
=
\psi^{*}(f_{2}^{-1}|_{\Loc(\Ed_{2,ij})})
\,\mathfrak{f}|_{\Ed_{1,ij}}(d)
=
f_{1}^{-1}|_{\Loc(\Ed_{1,ij})},
\]
where
$h_{\Loc(\Ed_{1,ij})}:=\mathfrak{f}|_{\Ed_{1,ij}}(1)$.
Thus
\[
(\psi^{*}(f_{2}^{-1})h^{d})|_{\Loc(\Ed_{1,ij})}
=
f_{1}^{-1}|_{\Loc(\Ed_{1,ij})}.
\]
Since the open subsets $\Loc(\Ed_{1,ij})$ cover $\overline{C}$, we conclude that
$\psi^{*}(f_{2}^{-1})h^{d}=f_{1}^{-1}$.

Conversely, suppose that
$\psi^{*}(f_{2}^{-1})h^{d}=f_{1}^{-1}$,
and define $\mathfrak{f}:=[1]\otimes h$. For $i,j\in\{0,\infty\}$, let
\[
U_{(i,j)}:=C_{i}\cap\psi^{-1}(C_{j}),
\qquad
V_{(i,j)}:=\psi(U_{(i,j)}).
\]
Let $\Ed_{1,ij}$ be the face of $\D_{1,i}$ with
$\Loc(\Ed_{1,ij})=U_{(i,j)}$, and let $\Ed_{2,ij}$ be the face of $\D_{2,j}$ with
$\Loc(\Ed_{2,ij})=V_{(i,j)}$, which exist by \cite[Proposition~3.13]{GaryAHS}.
Then
\[
\psi^{*}(f_{2}^{-1}|_{V_{(i,j)}})
\,h_{\Loc(U_{(i,j)})}^{d}
=
f_{1}^{-1}|_{U_{(i,j)}}.
\]
Since
$\mathfrak{f}|_{\Ed_{1,ij}}(d)=h|_{U_{(i,j)}}^{d}$,
it follows that
$(\psi,\mathfrak{f}|_{\Ed_{1,ij}}):\Ed_{2,ij}\to\Ed_{1,ij}$
is an isomorphism for every pair $(i,j)$. Therefore,
$(\Sd_{f_{1}},f_{1}^{-1})$
and
$(\Sd_{f_{2}},f_{2}^{-1})$
are isomorphic.
\end{proof}

\begin{lemma}\label{lem: integrality of principal phi pairs}
Let $f\in k(\overline{C})^{*}$. 
Let $\Sd_{f}$ be the divisorial fan defined by~\eqref{eq: divisorial fan Sdf}.
The $\Phi$-pair $(\Sd_{f},f^{-1})$ is integral if and only if there is no
$f'\in k(\overline{C})^{*}$ satisfying
$f'^{\,m}=f$
for some divisor $m>1$ of $d$.
\end{lemma}

\begin{proof}
Assume first that $(\Sd_{f},f^{-1})$ is integral. Then, by definition,
$(\D_{0},f^{-1}|_{C_{0}})$ is integral. Suppose that
$f'\in k(\overline{C})^{*}$ satisfies
$f'^{\,m}=f$ for some divisor $m>1$ of $d$. Restricting to $C_{0}$ yields
$(f'|_{C_{0}})^{m}=f|_{C_{0}}$.
Since $A[C_{0},\D_{0}]$ is normal and
$k(C_{0})\subset\mathrm{Quot}(A[C_{0},\D_{0}])$, we obtain
$f'|_{C_{0}}\in A[C_{0},\D_{0}]$.
Moreover,
$f'|_{C_{0}}$ is homogeneous of degree $d/m$,
contradicting the integrality of
$(\D_{0},f^{-1}|_{C_{0}})$ (see Definition~\ref{def: integral pairs}).

Conversely, assume that the $\Phi$-pair $(\Sd_{f},f^{-1})$ is not integral. Then there exists
$\D\in\Sd_{f}$ such that
$(\D,f^{-1}|_{\Loc(\D)})$
is not integral. Hence there exist
$f'\in A[\Loc(\D),\D]^{*}\cap
H^{0}(\Loc(\D),\mathscr{O}_{\Loc(\D)}(\D(d/m)))$
for some divisor $m>1$ of $d$ such that
$f'^{\,m}=f|_{\Loc(\D)}$.
Since
$A[\Loc(\D),\D]\subset k(\overline{C})$,
we may regard $f'$ as an element of
$k(\overline{C})^{*}$, and therefore
$f'^{\,m}=f$.
This proves the assertion.
\end{proof}

\subsection{Finiteness and infiniteness results} 
We begin with a preliminary result to compare the different $\Phi$-structures on a given divisorial fan.

\begin{lemma}\label{lemma: number of structures}
Two $\Phi$-structures on a divisorial fan $(\Sd,\overline{C})$ differ by a family of invertible regular functions $h_{\D}\in k[\Loc(\D)]^{*}$ such that
$h_{\D}|_{\Loc(\D\cap\Ed)}=h_{\Ed}|_{\Loc(\D\cap\Ed)}$
for any pair $\D,\Ed\in \Sd$. 

Moreover, they define $G$-isomorphic $G$-normal curves if and only if for every $\D$ there exists $f_{\D}\in k[\Loc(\D)]^*$ such that $h_{\D}=f_{\D}^{d}$; and $f_{\D}|_{\Loc(\D\cap\Ed)}=f_{\Ed}|_{\Loc(\D\cap\Ed)}$
for any pair $\D,\Ed\in \Sd$. 
\end{lemma}

\begin{proof}
Let $\beta_{1},\beta_{2}\in k(\overline{C})^{*}$ be two $\Phi$-structures on $\Sd$. For any pp-divisor $\D\in \Sd$, we obtain pairs $(\D,\alpha_{1,\D})$ and $(\D,\alpha_{2,\D})$ encoding the respective restrictions $\Phi_{\beta_{1}}\colon X(\D)\to T/G$ and $\Phi_{\beta_{2}}\colon X(\D)\to T/G$. Thus each quotient $h_{\D}:=\alpha_{1,\D}/\alpha_{2,\D}$ lies in $k[\Loc(\D)]^{*}$. Moreover, since for every $\D,\Ed\in \Sd$ and $i\in\{1,2\}$ we have
$\alpha_{i,\D}|_{\Loc(\D\cap\Ed)}=\alpha_{i,\Ed}|_{\Loc(\D\cap\Ed)}$,
it follows that $
h_{\D}|_{\Loc(\D\cap\Ed)}=h_{\Ed}|_{\Loc(\D\cap\Ed)}$.
Hence there exists a compatible family $h_{\D}\in k[\Loc(\D)]^{*}$ satisfying the above compatibility condition, which proves the first assertion.

The second assertion follows by applying Proposition~\ref{prop: pais D alpha} on each pp-divisor of the divisorial fan.
\end{proof}

The following proposition, which is the projective analogue of Proposition~\ref{prop: boundedness affine case}, is a special case of \cite[Corollary~3.3.7]{Bri26}, but here we give an independent proof using Altmann--Hausen--S\"u\ss{} theory. As in the affine case, the finiteness result holds for higher-dimensional $G$-varieties (see Remark~\ref{remark: finiteness in higher-dimensional G-varieties}). In addition, we obtain an upper bound for the number of $G$-isomorphism classes of $G$-normal curves over $\overline{C}$ with a fixed branch locus (i.e., the locus over which the curve fails to be a $G$-torsor), as in the affine case. 
By Proposition~\ref{proposition: reduction to three pp-div} and Theorem~\ref{th: main result AHS descritpion}, it suffices to count divisorial fans of the form $\Sd_{0,\infty}$ as in~\eqref{eq: simplified form of div fans over a curve}. By Remark~\ref{remark: reduction to the affine case}, this further reduces to counting the possible pp-divisors $\D_{0}$ occurring in such divisorial fans. We then obtain the following.

\begin{proposition}\label{proposition: finite number fixed branch locus}
As before, let $\overline{C}$ be a smooth projective curve, let $G=\mu_d$ be a finite subgroup scheme of $T=\Gm$, and assume that $k^{*}=(k^{*})^{d}$. Then there are only finitely many $G$-isomorphism classes of $G$-normal projective curves over $\overline{C}$ with branch locus $Z$.
Moreover, their number is bounded by
\small
\[
N:=
\min_{c\in\overline{C}}
\left|
(k[\overline{C}\setminus(\{c\}\cup Z)]^{*}/k[\overline{C}\setminus\{c\}]^{*})
\big/
(k[\overline{C}\setminus(\{c\}\cup Z)]^{*}/k[\overline{C}\setminus\{c\}]^{*})^{d}
\right|
\cdot
\left|
k[\overline{C}\setminus\{c\}]^{*}/
(k[\overline{C}\setminus\{c\}]^{*})^{d}
\right|.
\]
\normalsize
If, in addition, there exists a closed point $c\in\overline{C}$ such that
$C_{0}:=\overline{C}\setminus\{c\}$ satisfies that $k[C_{0}]$ is a UFD, then one obtains the simpler bound
\[
N':=(d-1)^{n-1}\cdot
\left|
k[C_{0}]^{*}/(k[C_{0}]^{*})^{d}
\right|,
\]
where $n$ is the cardinality of $|Z \cap C_0|$.
\end{proposition}

\begin{proof}
Fixing the branch locus of the quotient morphism $C \to C/G \simeq \overline{C}$ amounts to fixing the support of the associated divisorial fan $\Sd$ over $\overline{C}$ (see Remark~\ref{remark: correspondence branch locus}). 
By Proposition~\ref{proposition: reduction to three pp-div} and Remark~\ref{remark: reduction to the affine case}, it suffices to count the possible pp-divisors $\D_{0}$ occurring in a divisorial fan of the form $\Sd_{0,\infty}$ as in~\eqref{eq: simplified form of div fans over a curve} (in Section~\ref{sec: projective case}) and its respective $\Phi$-structures. Hence, the assertion follows from Proposition~\ref{prop: boundedness affine case}.
\end{proof}

In a slightly different direction, assuming from now on that the ground field $k$ is algebraically closed, we investigate when the sets $\E_n$ are finite or infinite. Recall from \eqref{eq: def of En} in the introduction that, for every integer $n\ge 0$, we denote
\[
\E_n:=
\left\{
\begin{array}{c}
\text{$G$-isomorphism classes of $G$-normal projective curves with quotient $\overline{C}$}\\
\text{such that $C\to\overline{C}$ is a $G$-torsor outside a branch locus of cardinal $n$}
\end{array}
\right\}.
\]

\begin{theorem}\label{theorem: finiteness of En}
Assume that the ground field $k$ is algebraically closed.
Let $G=\mu_{{d}} \subset \Gm$ be a finite subgroup scheme with ${{d}} \geq 2$, and recall that we denote by $\varphi$ the Euler's totient function.
Let $\overline{C}$ be a smooth projective curve of genus $g$. Then:
\begin{enumerate}
    \item\label{theorem: finiteness of En part g=0 n01} If $g = 0$ and $n \in \{ 0,1\}$, then $\E_n$ is empty.
    \item\label{theorem: finiteness of En part g=0 n23}   
    If $g=0$ and $n\in\{2,3\}$, then
$|\E_3| < (d-1)^{\,2}$ and 
$|\E_{2}|=\varphi(d)$.
    \item\label{theorem: finiteness of En part g=0 n>3} If $g=0$ and $n\ge4$, then $\E_n$ is infinite.
    \item\label{theorem: finiteness of En part g>1 n=0} If $g\ge1$ and $n=0$, then $  |\E_n|\le {{d}}^{2g}-1$.
    \item\label{theorem: finiteness of En part g>1 n=1} If $g\ge1$ and $n=1$, then $\E_n$ is empty. 
    \item\label{theorem: finiteness of En part g=1 n>2} If $g=1$ and $n\ge 2$, then $\E_n$ is infinite.
    \item\label{theorem: finiteness of En part g>1 n>2} If $g\ge 2$ and $n\ge 2g+2$, then $\E_n$ is infinite.
\end{enumerate}
\end{theorem}

\begin{proof}
\begin{itemize}
\item 
Proof of~\ref{theorem: finiteness of En part g=0 n01}.
The result is classical (see, for instance, the proof of \cite[Lemma 4.10]{FM26}).
Let $Y$ be a curve. Using the short exact sequence (for the \emph{fppf} topology) of \cite[Corollary III.4.5.5]{DG70}
\begin{equation} \label{eq: exact sequence GD}
0 \to \mathscr{O}_{Y}(Y)^*/(\mathscr{O}_{Y}(Y)^*)^{{{d}}}
   \to \H^1(Y,\mu_{{{d}}})
   \to \Pic(Y)[{{{d}}}]
   \to 0,
\end{equation}
we obtain $\H^1(Y,\mu_{{{d}}}) \simeq \Pic(Y)[{{{d}}}]$
whenever $Q_Y:=\mathscr{O}_{Y}(Y)^*/(\mathscr{O}_{Y}(Y)^*)^{{{d}}}$
is trivial. Equivalently, the isomorphism classes of $\mu_d$-torsors over $Y$ are in bijection with the isomorphism classes of line bundles $\L$ on $Y$ such that $\L^{\otimes d}\simeq \O_Y$.
Now, for $Y=\P^1$ or $\A^1$, one has $Q_Y=\{1\}$, while
$\Pic(\P^1)\simeq \Z$ and $\Pic(\A^1)=\{0\}$.
In particular, $\Pic(Y)[d]=\{0\}$ in both cases, so every $\mu_d$-torsor $X \to Y$ with $Y=\P^1$ or $\A^1$ is trivial. This contradicts the assumption that $X$ is a curve (hence irreducible and reduced). Therefore, $\E_n=\varnothing$ for $n\in\{0,1\}$.

\smallskip

\item 
Proof of~\ref{theorem: finiteness of En part g=0 n23}.  For $n=3$, any integral $\Phi$-pair $(\Sd,\beta)$ over $\P^{1}$ such that $[(\Sd,\beta)] \in \mathcal{S}_{3}$ can be assumed to have support given by $\mathrm{Supp}(\Sd,\P^{1})=\{0,1,\infty\}$. 
Indeed, any three points in $\P^{1}$ can be sent to $\{0,1,\infty\}$ by an automorphism of $\P^{1}$, which also move the $\Phi$-structure (see Remark~\ref{remark: reduction by Aut(C)}). Besides, since $k[C_0]^{*}=k[C_{\infty}]^{*}=k^{*}$, all $\Phi$-pairs $(\Sd,\beta')$ define the same $G$-normal curve (up to $G$-isomorphism) by Lemma~\ref{lemma: number of structures}. Hence, it suffices to count the number of divisorial fans $\Sd$ supported on $\{0,1,\infty\}$ and admitting a $\Phi$-structure. By Proposition~\ref{proposition: admissible structure}, these divisorial fans have pp-divisors (see \eqref{eq: simplified form of div fans over a curve} in Section~\ref{sec: projective case} for notational details)
\[
\D_{0}:=\left[\frac{r_{0}}{{{d}}_{0}}\right]\otimes\{0\}+\left[\frac{r_{1}}{{{d}}_{1}}\right]\otimes\{1\}+\varnothing\otimes\{\infty\}\quad
\mathrm{and} \quad \D_{0}:=\varnothing\otimes\{0\}+\left[\frac{r_{1}}{{{d}}_{1}}\right]\otimes\{1\}+\left[-\frac{r_{0}}{{{d}}_{0}}-\frac{r_{1}}{{{d}}_{1}}\right]\otimes\{\infty\}.
\] 
This reduces the classification to the choice of the pp-divisor $\D_{0}$, since $\D_{\infty}$ is then determined by $\D_{0}$. 
Thus, by Proposition~\ref{prop: boundedness affine case}, the number of possible polyhedra gives $\lvert \E_{3} \rvert=\lvert \mathcal{S}_{3} \rvert \le(d-1)^{2}$. 
Following Remark~\ref{remark: boundedness affine case}, the equality never holds since $\mathrm{lcm}(d_{0},d_{1})=d$ is a necessary condition.

Similarly, for $n=2$ and taking $\mathrm{Supp}(\Sd,\P^{1})=\{0,\infty\}$, Proposition~\ref{proposition: admissible structure} implies that any divisorial fan supported on $\{0,\infty\}$ admitting a $\Phi$-structure has pp-divisors
\[
\D_{0}
=
\left[\frac{r_{0}}{{{d}}_{0}} \right]\otimes\{0\}
+
\varnothing\otimes\{\infty\},
\qquad
\D_{\infty}
=
\varnothing\otimes\{0\}
+
\left[-\frac{r_{0}}{{{d}}_{0}}\right]\otimes\{\infty\}.
\] 
Moreover, from Remark~\ref{remark: UFD equivalence}, we can assume that $\mathrm{gcd}(r_{0},d_{0})=1$ and $0<r_{0}<d_{0}$. This yields $\lvert \E_{2} \rvert = \lvert \mathcal{S}_{2} \rvert \le \varphi({{{d}_{0}}})$. Let $(\Sd,\beta)$ be a $\Phi$-pair over $\overline{C}$. By definition, $(\Sd,\beta)$ is an integral $\Phi$-pair if and only if the induced $\Phi$-pairs $(\D_{0},\alpha_{0})$ and $(\D_{\infty},\alpha_{\infty})$ are integral. Since $\Loc(\D_{0})\simeq \Loc(\D_{\infty})\simeq \A^{1}$, the $\Phi$-pairs $(\D_{0},\alpha_{0})$ and $(\D_{\infty},\alpha_{\infty})$ are integral if and only if $d_{0}=d$ by Corollary~\ref{corollary: integral pair non non-constant regular affine}. Hence, $\lvert \E_{2} \rvert = \lvert \mathcal{S}_{2} \rvert \le \varphi({{d}})$ and the equality holds. Indeed, let $0<r_{0}<{{d}}$ be such that $\gcd(r_{0},{{d}})=1$. Let us denote by $\Sd_{r_{0}}$ the divisorial fan corresponding to $r_{0}$. By Proposition~\ref{proposition: integral divisorial pair} and Corollary~\ref{corollary: integral pair non non-constant regular affine}, any $\Phi$-pair $(\Sd_{r_{0}},\beta)$ is integral and defines a $G$-normal curve by Theorem~\ref{th: main result AHS descritpion}. We thus obtain $\lvert \E_{2} \rvert = \varphi({{d}})$.

\smallskip

\item  
Proof of~\ref{theorem: finiteness of En part g=0 n>3}. As in the previous case, all $\Phi$-structures are isomorphic over $\mathbb{P}^{1}$. Therefore, for $n \geq 4$, up to applying an automorphism of $\mathbb{P}^{1}$, it suffices to count the number of divisorial fans $\Sd$ over $\overline{C}=\P^1$ supported on $\{0,1,\infty, p_1,\ldots,p_{n-3}\}$ and admitting a $\Phi$-structure. 
Following the conventions of \eqref{eq: simplified form of div fans over a curve} in Section~\ref{sec: projective case}, such a divisorial fan satisfies 
\[
\frac{r_{0}}{d_{0}}+\frac{r_{1}}{d_{1}}+\cdots+\frac{r_{e}}{d_{e}}+\frac{r_{\infty}}{d_{\infty}}=0,
\]
by Proposition~\ref{proposition: admissible structure}. Moreover, Proposition~\ref{proposition: integral divisorial pair} and Corollary~\ref{corollary: integral pair non non-constant regular affine} imply that a $\Phi$-pair $(\Sd,\beta)$ is integral if and only if at least one of the integers ${{d}}_{i}$ is equal to ${{d}}$ on each $\D_{0}$ and $\D_{\infty}$. Hence, for any given support, we can always find a divisorial fan $\Sd$ with that support that admits a $\Phi$-structure and for which every $\Phi$-pair is integral.
Note that each additional point $p_{i} \in \P^1 \setminus \{0,1,\infty \}$ introduces a continuous parameter in the choice of the divisorial fan. Since every element $\psi \in \mathrm{Aut}(\mathbb{P}^{1})=\PGL_{2}$ is uniquely determined by the images of three points,  there are infinitely many non-isomorphic such divisorial fans, and hence $\E_n \simeq \mathcal{S}_{n}$ is infinite.

\smallskip

\item 

Proof of~\ref{theorem: finiteness of En part g>1 n=0}. 
Since $\E_0$ is the set of isomorphism classes of integral $\mu_{{{d}}}$-torsors over $\overline{C}$, it follows that $\E_0 \subsetneq \H^1(\overline{C},\mu_{{{d}}})$, as the trivial $\mu_{{{d}}}$-torsor is not a curve (it is disconnected or non-reduced).
Using the exact sequence \eqref{eq: exact sequence GD}, 
we obtain $\H^1(\overline{C},\mu_{{{d}}}) \simeq \Pic(\overline{C})[{{{d}}}]$ and hence an injection $\E_0 \hookrightarrow \Pic(\overline{C})[{{{d}}}]$. The classical short exact sequence (see \cite[\href{https://stacks.math.columbia.edu/tag/03RN}{Tag 03RN}]{stacks-project})
\[
0 \to \mathrm{Jac}(\overline{C})
   \to \mathrm{Pic}(\overline{C})
   \xrightarrow{\deg} \mathbb{Z}
   \to 0
\]
yields
$\E_0 \hookrightarrow \mathrm{Jac}(\overline{C})[d]$.
The latter is the ${{{d}}}$-torsion subgroup of the group of $k$-points of the abelian variety $\mathrm{Jac}(\overline{C})$, hence finite of cardinality at most ${{{d}}}^{2g}$, by \cite[\href{https://stacks.math.columbia.edu/tag/03RP}{Tag 03RP}]{stacks-project}. 
Thus $|\E_0|\le {{{d}}}^{{2g}}-1$ (we always exclude the trivial $\mu_d$-torsor).

\smallskip

\item 
Proof of~\ref{theorem: finiteness of En part g>1 n=1}. 
The problem reduces to showing that every $\mu_d$-torsor on the punctured curve $C_0:=\overline{C}\setminus\{c\}$ extends uniquely to a $\mu_d$-torsor on $\overline{C}$.
Since
\[
\mathscr O_{\overline C}(\overline C)^\times/
\bigl(\mathscr O_{\overline C}(\overline C)^\times\bigr)^d
=\{1\}=
\mathscr O_{C_0}(C_0)^\times/
\bigl(\mathscr O_{C_0}(C_0)^\times\bigr)^d,
\]
the exact sequence \eqref {eq: exact sequence GD} yields canonical identifications
\[
\H^1(\overline C,\mu_d)\simeq \Pic(\overline C)[d],
\qquad
\H^1(C_0,\mu_d)\simeq \Pic(C_0)[d].
\]
On the other hand, restriction induces an exact sequence
\[
\mathbb Z \longrightarrow \Pic(\overline C)
\xrightarrow{\pi}
\Pic(C_0)\longrightarrow 0,
\]
where $1\in\mathbb Z$ maps to the class of $(c)$.
Applying the snake lemma to multiplication by $d$ gives 
\[
0\longrightarrow \Pic(\overline C)[d]
\longrightarrow \Pic(C_0)[d]
\xrightarrow{\delta}
\mathbb Z/d\mathbb Z .
\]
It therefore suffices to show that $\delta=0$. Let
$(x) \in \Pic(C_0)[d]$ and choose a lift
$(\widetilde x) \in \Pic(\overline C)$ with $\pi((\widetilde x))=(x)$.
Since $d\cdot(x)=0$, we have
$d\cdot(\widetilde x)\in \ker(\pi)$, hence
$d\cdot(\widetilde x)=n\cdot(c)$
for some $n\in\mathbb Z$.
Taking degrees, we obtain $d\cdot\deg((\widetilde x))=n$,
so $d$ divides $n$. By definition, $\delta((x))=\overline n\in \mathbb Z/d\mathbb Z$,
and therefore $\delta((x))=\overline{0}$.
Thus $\delta$ is trivial, and the above exact sequence yields
$\Pic(\overline C)[d]\simeq \Pic(C_0)[d].$
Equivalently, $\H^1(\overline C,\mu_d)\simeq \H^1(C_0,\mu_d)$,
which proves the claim.

\smallskip

\item
Proof of~\ref{theorem: finiteness of En part g=1 n>2}. 
Assume that $g=1$ and $n\ge 2$. We show that $\E_n$ is infinite. \\
Since $\overline{C}$ is an elliptic curve, fix a point
$0\in\overline{C}$ as the identity of its group law.
Let $c_{1},\dots,c_{n}$ be pairwise distinct points of $\overline{C}$.
A divisor
\[
D:=m_{1}\{c_{1}\}+\cdots+m_{n}\{c_{n}\}
\]
is principal if and only if
\[
m_{1}+\cdots+m_{n}=0
\quad\text{and}\quad
m_{1}c_{1}+\cdots+m_{n}c_{n}=0
\in J(\overline{C})\simeq\overline{C}.
\]
Assume that $n=2$. Then $D$ is principal if and only if
$m_{1}+m_{2}=0$ and
$m_{1}c_{1}+m_{2}c_{2}=0\in\overline{C}$.
Hence $m_{2}=-m_{1}$, and the latter condition becomes
\[
m_{1}(c_{1}-c_{2})=0\in\overline{C}.
\]
Since $c_{1}\neq c_{2}$, we have 
$m_{1}\ge2$.
Thus the set of principal divisors with support of cardinality $2$ is
\[
\mathcal{Z}_{2}
=
\{m\{c\}-m\{c+e\}\mid
m\ge2,\ (c,e)\in\overline{C}\times\overline{C}[m]\}.
\]
By Lemma~\ref{lem: principal phi pairs}, the divisors
$D_{1}=m\{c\}-m\{c+e\}$ and
$D_{2}=m\{0\}-m\{e\}$ define isomorphic $\Phi$-pairs.
Hence it suffices to consider the set
\[
\mathcal{Z}_{2}(0)
:=
\{m\{0\}-m\{e\}\mid
m\ge2,\ e\in\overline{C}[m]\}.
\]
Now let
\[
D_{r,e}
:=
(dr+1)\{0\}-(dr+1)\{e\},
\qquad r\ge1,\ e\in\overline{C}[m].
\]
We claim that these divisors give rise to infinitely many pairwise non-isomorphic $\Phi$-pairs.
Indeed, by Lemma~\ref{lem: principal phi pairs}, if
$D_{1},D_{2}\in\mathcal{Z}_{2}(0)$ define isomorphic $\Phi$-pairs, then there exists
$\psi\in\Aut(\overline{C})$ such that every coefficient of
$D_{1}-\psi^{*}D_{2}$ is divisible by $d$.
Thus, if
$D_{r,e}$ and $D_{r',e'}$ define isomorphic $\Phi$-pairs, then the coefficients of
\[
D_{r,e}-\psi^{*}D_{r',e'}
=
d(r-r')\{0\}
-(dr+1)\{e\}
+(dr'+1)\{\psi(e')\}
\]
are divisible by $d$.
Since $dr+1\equiv dr'+1\equiv1\pmod d$, this is possible only if
$\psi(e')=e$.
As there are only finitely many automorphisms fixing the origin, the divisors
$D_{r,e}$ yield infinitely many pairwise non-isomorphic $\Phi$-pairs.\\
It remains to prove that these $\Phi$-pairs are integral.
Let $f_{r,e}\in k(\overline{C})^{*}$ satisfy
$\div(f_{r,e})=D_{r,e}$.
Suppose, for contradiction, that
$(\Sd_{f_{r,e}},f_{r,e}^{-1})$
is not integral.
By Lemma~\ref{lem: integrality of principal phi pairs},
there exist
$f'\in k(\overline{C})^{*}$ and a divisor
$m>1$ of $d$ such that
$f'^{\,m}=f_{r,e}$.
Hence
\[
m\,\div(f')
=
\div(f_{r,e})
=
(dr+1)\{0\}
-(dr+1)\{e\},
\]
so $m\mid(dr+1)$.
Since $\gcd(dr+1,d)=1$, this contradicts
$m\mid d$.
Therefore,
$(\Sd_{f_{r,e}},f_{r,e}^{-1})$
is integral for every $r$ and $e$, and consequently
$\E_{2}$ is infinite.

Assume now that $n\ge3$. We consider principal divisors of the form
\[
D=\{c_{1}\}+\cdots+\{c_{n-1}\}-(n-1)\{c_{n}\},
\]
constructed as follows. If $n=3$, choose distinct points
$c_{1},c_{2}\in\overline{C}$ such that
$c_{1}+c_{2}$ is non-torsion, and let
$c_{3}$ be any solution of
$2x=c_{1}+c_{2}$.
If $n\ge4$, choose pairwise distinct non-torsion points
$c_{1},\dots,c_{n-2}\in\overline{C}$, and define
\[
c_{n}:=c_{1}+\cdots+c_{n-2},
\qquad
c_{n-1}:=(n-2)c_{n}.
\]
As above, it suffices to consider divisors with $c_{1}=0$. Thus we may restrict to automorphisms fixing the origin, so the orbit
\[
\{\psi^{*}D\mid\psi\in\Aut_{0}(\overline{C})\}
\]
is finite.
We claim that these divisors give rise to infinitely many pairwise non-isomorphic $\Phi$-pairs.
Indeed, by Lemma~\ref{lem: principal phi pairs}, if two such divisors
$D_{1}$ and $D_{2}$ define isomorphic $\Phi$-pairs, then every coefficient of
$D_{1}-\psi^{*}D_{2}$ is divisible by $d$ for some
$\psi\in\Aut_{0}(\overline{C})$.
This is possible only if
$D_{1}=\psi^{*}D_{2}$.
Otherwise, there exists a point $c_{i,2}$ appearing in
$\psi^{*}D_{2}$ but not in $D_{1}$, so that
$D_{1}-\psi^{*}D_{2}$ has a coefficient equal to~$1$, a contradiction.
Hence these divisors yield infinitely many pairwise non-isomorphic $\Phi$-pairs.
Finally, if $\div(f)=D$, then
$(\Sd_{f},f^{-1})$
is integral by Lemma~\ref{lem: integrality of principal phi pairs}, since $D$ has coefficients equal to~$1$.

\smallskip

\item 
Proof of~\ref{theorem: finiteness of En part g>1 n>2}. 
Assume that $g\ge2$ and $n\ge2g+2$. We show that $\E_n$ is infinite.\\
Fix a point $c\in\overline{C}$. Since
$\deg((n-1)c)=n-1\ge2g+1$,
the divisor $(n-1)c$ is very ample (\cite[Ch.~IV, Cor.~3.2]{Har}). Thus, Bertini's theorem implies that a dense open subset of the complete linear system $|(n-1)c|$ parametrizes reduced divisors. Intersecting with the complement of the hyperplane consisting of divisors containing $c$, we obtain a dense open subset parametrizing divisors
$D=c_1+\cdots+c_{n-1}$,
where the $c_i$ are pairwise distinct and different from $c$. Consequently, there exist infinitely many principal divisors of the form
$D-(n-1)c$,
whose support consists of exactly $n$ distinct points.
By Lemma~\ref{lem: principal phi pairs}, if two divisors
\[
E_{1}:=D_{1}-(n-1)c_{1}
\qquad\text{and}\qquad
E_{2}:=D_{2}-(n-1)c_{2}
\]
define isomorphic $\Phi$-pairs, then $d$ divides every coefficient of
$E_{1}-\psi^{*}E_{2}$ for some $\psi\in\Aut(\overline{C})$.
This is possible only if
$E_{1}=\psi^{*}E_{2}$, since otherwise
$E_{1}-\psi^{*}E_{2}$ has a coefficient equal to~$1$.
Hence the divisors $D-(n-1)c$ give rise to infinitely many pairwise non-isomorphic $\Phi$-pairs, because $\Aut(\overline{C})$ is finite for $g\ge2$~\cite{Sch38}.
Moreover, since $D-(n-1)c$ has a coefficient equal to~$1$, the corresponding $\Phi$-pair is integral by Lemma~\ref{lem: integrality of principal phi pairs}. Therefore, $\E_n$ is infinite.
\end{itemize}
\end{proof}

\subsection{A complete description of \texorpdfstring{$\E_{2}$}{E2} over \texorpdfstring{$\P^1$}{P1}} \label{sec: complete description of E2 and E3}
In this section, we describe the set $\E_2$ under the assumption that the ground field $k$ is algebraically closed.
We adopt the convention 
\[\Spec(k[u])=U_{0}:=\{[z_{0},z_{1}]\in \P^1 \mid z_{1}\neq 0\}\ \ \text{and}\ \ \Spec(k[w])=U_{\infty}:=\{[z_{0},z_{1}]\in \P^1 \mid z_{0}\neq 0\},\] with gluing given by the canonical morphisms $k[u]\to k[u,w]/(uw-1)$ and $k[w]\to k[u,w]/(uw-1)$. 

By Proposition~\ref{proposition: admissible structure}, the divisorial fan
\[
\Sd:=\{\D_{0},\D_{\infty},\D_{0}\cap\D_{\infty}\} \in \mathcal{S}_2 \simeq \mathcal{E}_2,
\]
with
\[ 
\D_{0} := \left[\frac{r_{0}}{d_{0}}\right]\otimes\{0\}+\varnothing\otimes\{\infty\},
\text{ and }
\D_{\infty}:=\varnothing\otimes\{0\}
+\left[-\frac{r_{0}}{d_{0}}\right]\otimes\{\infty\},
\]
is the unique one admitting a $\Phi$-structure (unique up to isomorphism by Lemma~\ref{lemma: number of structures}). Moreover, by Lemmas~\ref{lem: fiber polyhedrons} and~\ref{lem: bounded ry} and Remark~\ref{remark: UFD equivalence}, we may assume
\[
0<r_{0}<{{d}}_{0},\qquad \gcd(r_{0},{{d}}_{0})=1,\qquad {{d}}_{0}\mid {{d}}.
\]
Finally, Proposition~\ref{proposition: integral divisorial pair} together with Corollary~\ref{corollary: integral pair non non-constant regular affine} imply that ${{d}}_{0}$ is equal to $d$. Besides, by Theorem~\ref{theorem: finiteness of En}~\ref{theorem: finiteness of En part g=0 n23}, there are exactly $\varphi({{d}})$ pairwise non-isomorphic $G$-normal curves in this case. The corresponding divisorial fans are parametrized by integers $0<r<{{d}}$ with $\gcd(r,{{d}})=1$, and are given by
\[
\Sd_{r}
:=
\Bigl\{
\underbrace{\left[\frac{r}{{{d}}}\right]\otimes\{0\}+\varnothing\otimes\{\infty\}}_{\D_{0}},
\;
\underbrace{\varnothing\otimes\{0\}+\left[-\frac{r}{{{d}}}\right]\otimes\{\infty\}}_{\D_{\infty}},
\;
\underbrace{\varnothing\otimes\{0\}+\varnothing\otimes\{\infty\}}_{\D_{0}\cap\D_{\infty}}
\Bigr\}  \in \mathcal{S}_2 \simeq \mathcal{E}_2.
\]

We follow the conventions of Section~\ref{sec: Explicit construction of XD from a pp-divisor and examples}. The algebras $A[C_{0},\D_{0}]$ and $A[C_{\infty},\D_{\infty}]$ are generated by
\[
x:=\mathfrak{X}_{1}, \quad
x_{l_{s}}:=
\frac{\mathfrak{X}_{l_{s}}}{
u^{s}},
\quad
a:=
\frac{\mathfrak{X}_{{{d}}}}{
u^{r}
},
\]
and
\[
y:=\mathfrak{X}_{-1}, \quad
y_{l_{s}}:=
\frac{\mathfrak{X}_{-l_{s}}}{
w^{s}},
\quad
a:=
w^{r}\mathfrak{X}_{{{d}}
},
\]
respectively. The following lemma shows that each of these algebras is in fact generated by only four of these elements. It plays a central role in the construction of the algebras associated with an integral $\Phi$-pair $(\D,\alpha)$ over a smooth affine curve $C_{0}$.

\begin{lemma}
Let $r,d$ be positive integers with $\gcd(r,d)=1$, and set
$l_s:=\min\{l\in\N\mid lr\ge sd\}$.
If $r>1$, then there exists $1\le j<r$ such that $l_jr=jd+1$. Moreover, for every $s\ge1$, there exist positive integers $M_s,N_s$ satisfying
$M_sj=N_sr+s$ and $M_sl_j=N_sd+l_s$.
\end{lemma}

\begin{proof}
Since $r$ and $d$ are coprime, there exist positive integers $A$ and $B$ such that $Ar=Bd+1$. Writing $B=rQ+j$ with $0<j<r$, we obtain $A=l_j$, and hence $l_jr=jd+1$. In particular, $\gcd(j,r)=1$, so for every $s\ge1$ there exist positive integers $\tilde M_s,\tilde N_s$ such that
\[
\tilde M_sj=\tilde N_sr+s.
\]
The general solution is given by
\[
M_s(\lambda)=\lambda r+\tilde M_s,\qquad
N_s(\lambda)=\lambda j+\tilde N_s.
\]
It remains to choose $\lambda$ so that $M_s(\lambda)l_j=N_s(\lambda)d+l_s$. This amounts to
\[
\lambda(l_jr-jd)=l_s-\tilde M_sl_j+\tilde N_sd.
\]
Since $l_jr-jd=1$, taking
\[
\lambda=l_s-\tilde M_sl_j+\tilde N_sd
\]
yields the desired equality.
\end{proof}

When $r=1$, the algebra is generated by $u$, $x$, $a$, and $b=a^{-1}$. Assume now that $r>1$. By the previous lemma, there exists $j$ such that $l_jr=jd+1$. Moreover, for every $s\ge1$, there exist integers $M_s,N_s$ such that $M_sl_j=N_sd+l_s$, and hence
\[
x_{l_{j}}^{M_{s}}=\left(\frac{\mathfrak{X}_{l_{j}}}{
u^{j}}\right)^{M_{s}}=\frac{\mathfrak{X}_{M_{s}l_{j}}}{
u^{M_{s}j}}=\frac{\mathfrak{X}_{N_{s}d+l_{s}}}{
u^{N_{s}r+s}}=\frac{\mathfrak{X}_{N_{s}d}}{
u^{N_{s}r}}\frac{\mathfrak{X}_{l_{s}}}{
u^{s}}=\left(\frac{\mathfrak{X}_{d}}{
u^{r}}\right)^{N_{s}}\frac{\mathfrak{X}_{l_{s}}}{
u^{s}}=a^{N_{s}}x_{l_{s}}.
\]
Since $a$ is invertible, it follows that $x_{l_s}=x_{l_j}^{M_s}b^{N_s}$. Furthermore,
\[
x_{l_j}^{r}a^{-j}=x,
\qquad
ua^{l_j}=x_{l_j}^{d}.
\]
Thus, the algebra is generated by $u$, $x_{l_j}$, $a$, and $b$, and is given by
\[
A[C_{0},\D_{0}]:=\frac{k[u,x_{l_j},a,b]}{(x_{l_{j}}^{d}-ua^{l_{j}},ab-1)}.
\]

Similarly,
\[
A[C_{\infty},\D_{\infty}]:=\frac{k[w,y_{l_j},a,b]}{(w-y_{l_{j}}^{d}a^{l_{j}},ab-1)}.
\]

 Let $C_{r}\subset X(\Sd_{r})$ denote the fiber of $\Phi_{\beta}$ over the base point. The curve $C_{r}$ is obtained by gluing
\[
U_0
\simeq
\Spec\!\left(
\frac{k[u,x_{l_j}]}
{(x^{l_j}-u)}
\right)
\simeq
\Spec\!\left(
k[x_{l_j}]
\right)
\]
and
\[
U_\infty
\simeq
\Spec\!\left(
\frac{k[w,y_{l_j}]}
{(y^{l_j}-w)}
\right)
\simeq
\Spec\!\left(
k[y_{l_j}]
\right)
\]
via the relations $x_{l_j}y_{l_j}=1$. 
Since $x_{l_j}$ and $y_{l_j}$ have degrees $l_j$ and $-l_j$, respectively, it follows that
$C_r\simeq \mathbb{P}^1$
endowed with the $G$-action
$t\cdot[z_0:z_1]=[t^{l_j}z_0:z_1]$.
Since $\mathbb{P}^1$ is smooth, it is $G$-normal.

\bibliographystyle{abbrv}
\bibliography{biblio}

\end{document}